\documentclass[11pt]{article}

\usepackage{amsmath,amssymb}
\usepackage{verbatim}

\usepackage{times}
\usepackage{dsfont}
\usepackage{bm}
\usepackage{natbib}
\usepackage{algorithm}
\usepackage{xargs}[2008/03/08]
\usepackage{amssymb}
\usepackage{mathrsfs}
\usepackage{graphicx}
\usepackage{rotating,subfigure}
\usepackage[flushleft]{threeparttable}
\usepackage{multirow}
\usepackage{enumitem} 

\usepackage{star}

\allowdisplaybreaks[4]

\usepackage[colorlinks,
linkcolor=blue,
anchorcolor=blue,
citecolor=blue
]{hyperref}

\def\##1\#{\begin{align}#1\end{align}}
\def\$#1\${\begin{align*}#1\end{align*}}

\newcommand{\trace}{{\rm tr}}

\def\T{{ \mathrm{\scriptscriptstyle T} }}

\newcommand{\new}{{\rm new}}

\newcommand{\iid}{i.i.d.\,}
\newcommand{\SNR}{{\rm SNR}}
\newcommand{\as}{{\rm a.s.}}

\newcommand{\Rom}[1]{\text{\uppercase\expandafter{\romannumeral #1\relax}}}

\newcommand{\scolor}[1]{{\color{magenta}#1}}

\newcommand{\scomment}[1]{\scolor{$\dagger$}\marginpar{\tiny\scolor{S:\ #1}}}

\newcommand{\ycolor}[1]{{\color{blue}#1}}

\usepackage{txfonts}

\usepackage{geometry}
\begin{document}


\title{\LARGE  Sketched ridgeless linear regression: The role of downsampling} 

\author{Xin Chen
\thanks{The first two authors contributed equally.}
\thanks{Department of Operations Research and Financial Engineering, Princeton University, 98 Charlton St, Princeton, NJ 08544, USA; E-mail: \texttt{xc5557@princeton.edu}.} \and  Yicheng Zeng$^*$\footnotemark[3]\thanks{Shenzhen Research Institute of Big Data, the Chinese University of Hong Kong, 2001 Longxiang Boulevard, Shenzhen, Guangdong, China; E-mail:\texttt{statzyc@sribd.cn}. Yicheng Zeng was a postdoctoral fellow at the University of Toronto when the bulk of this work was done. } \and  Siyue Yang\footnotemark[4] \and   Qiang Sun\thanks{Department of Statistical Sciences, University of Toronto, 700 University Ave, Toronto, ON M5G 1X6, Canada; E-mail: \texttt{syue.yang@mail.utoronto.ca}, \texttt{qiang.sun@utoronto.ca}.}}



\date{}

\maketitle

\vspace{-0.25in}

\begin{abstract}


Overparametrization often helps improve the generalization performance. This paper presents a dual view of overparametrization suggesting that downsampling may also help generalize. Focusing on the proportional regime $m\asymp n \asymp p$, where $m$ represents the sketching size, $n$ is the sample size, and $p$ is the feature dimensionality, we investigate two out-of-sample prediction risks of the sketched ridgeless least square estimator. Our findings challenge conventional beliefs by showing that downsampling does not always harm generalization but can actually improve it in certain cases. We identify the optimal sketching size that minimizes out-of-sample prediction risks and demonstrate that the optimally sketched estimator exhibits stabler risk curves, eliminating the peaks of those for the full-sample estimator. To facilitate practical implementation, we propose an empirical procedure to determine the optimal sketching size. Finally, we extend our analysis to cover central limit theorems and misspecified models. Numerical studies strongly support our theory.

\end{abstract}

\tableofcontents

\section{Introduction}

According to international data corporation, worldwide data will grow to 175 zettabytes by 2025, with as much of the data residing in the cloud as in data centers. These massive datasets hold tremendous potential to revolutionize operations and analytics across various domains. However, their sheer size presents unprecedented computational challenges, as many traditional statistical methods and learning algorithms struggle to scale effectively. 


In recent years,  sketch-and-solve methods, also referred to as sketching algorithms, have emerged as a powerful solution for approximate computations over large datasets  \citep{pilanci2016fast, mahoney2011randomized}. Sketching algorithms first employ random sketching/projection or random sampling techniques to construct a small ``sketch" of the full dataset, and then use this sketch as a surrogate to perform analyses of interest that would otherwise be computationally impractical on the full dataset.

This paper focuses on the linear regression problem. We assume that we have collected a set of independent and identically distributed (i.i.d.) data points following the model:
\begin{equation}\label{eq:model}
y_i = \beta^\top x_i + \varepsilon_i,\ i = 1,\cdots,n,
\end{equation}
where $y_i \in \mathbb{R}$ represents the label of the $i$-th observation, $\beta \in \mathbb{R}^p$ is the unknown random regression coefficient vector, 
$\mathbb{R}^p \ni x_i\sim x \sim P_x$ is the $p$-dimensional feature vector of the $i$-th observation, with $P_x$ denoting a probability distribution on $\mathbb{R}^p$ having mean $\mathbb{E}(x) = 0$ and covariance $\cov(x) = \Sigma$. The $i$-th random noise term $\mathbb{R} \ni\varepsilon_i \sim \varepsilon \sim P_\varepsilon$ is independent of $x_i$, with $P_\varepsilon$ being a probability distribution on $\RR$ having mean $\EE(\varepsilon) = 0$ and variance $\var(\varepsilon) = \sigma^2$. In matrix form, the model can be expressed as: 
\$
Y= X\beta + E,
\$
where $X = (x_1,\cdots,x_n)^\top \in \RR^{n\times p}$, $Y = (y_1,\cdots,y_n)^\top \in \mathbb{R}^n$, and $E = (\varepsilon_1,\cdots,\varepsilon_n)^\top \in \mathbb{R}^n$. 

We consider the following ridgeless least square estimator 
\$
\hat \beta := (X^\top X)^{+}X^\top Y = \lim_{\lambda \rightarrow 0^+}(X^\top  X + n\lambda I_p)^{-1} {X}^\top  {Y}, 
\$
where $(\cdot)^+$ denotes the Moore-Penrose pseudoinverse and $I_p \in \mathbb{R}^{p \times p}$ is the identity matrix. In the case where $\text{rank}(X) = p$, the estimator $\hat{\beta}$ reduces to the ordinary least square (OLS) estimator, which is the de-facto standard for linear regression due to its optimality properties.  However, computing the OLS estimator, typically done via QR decomposition \citep{golub2013matrix}, has a computational complexity of $\mathcal{O}(np^2)$. This renders the computation of the full-sample OLS estimator infeasible when the sample size $n$ and dimensionality $p$ reach the order of millions or even billions.

Sketching algorithms provide a solution to reduce the computational burden by reducing the data size, aka downsampling.  This is achieved by multiplying the full dataset $(X, Y)$ with a sketching matrix $S\in\RR^{m\times n}$ to obtain the sketched dataset $(SX, SY)\in \RR^{m\times p}\times \RR^m$, where $m < n$ is the sketching size. Instead of computing the full-sample OLS estimator, we compute the sketched ridgeless least square estimator based on the sketched dataset:
\#\label{eq:sketched}
\hat{\beta}^{S} 
&= (X^\top S^\top S X)^+ X^\top S^\top S Y  
= \lim_{\lambda\to 0^+} (X^\top S^\top SX + n\lambda I_p)^{-1}X^\top S^\top SY.
\#
The total computational complexity for computing the sketched data and sketched least square estimator is approximately $\cO(np\log m + m p^2)$ when using fast orthogonal sketches \citep{pilanci2016fast}. The prevailing belief is that sketching reduces the runtime complexity at the expense of statistical accuracy \citep{woodruff2014sketching,raskutti2016statistical, drineas2018lectures, dobriban2018asymptotics}. Indeed, as pointed out by \cite{dobriban2018asymptotics}, a larger number of samples leads to a higher accuracy. They showed that, in the case of orthogonal sketches, if one sketches to $m$ samples such that $p<m<n$, the test error increases by a factor of $m(n-p)/{n(m - p)}>1$, which equals $1.1$ when $m=10^6$, $n=10^7$, and $p = 10^5$. \cite{raskutti2016statistical} reported a similar phenomenon by considering the regime $n \gg p$ and various error criteria. However, these results only focus on the underparameterized regime ($p<m$) and do not reveal the statistical role of downsampling in a broader regime. It is therefore natural to ask the following questions:
\begin{quote}
\textit{What is the statistical role of downsampling? Does downsampling always hurt the statistical accuracy?}
\end{quote}

\begin{figure}[t]
    \centering  \includegraphics[width=.6\textwidth]{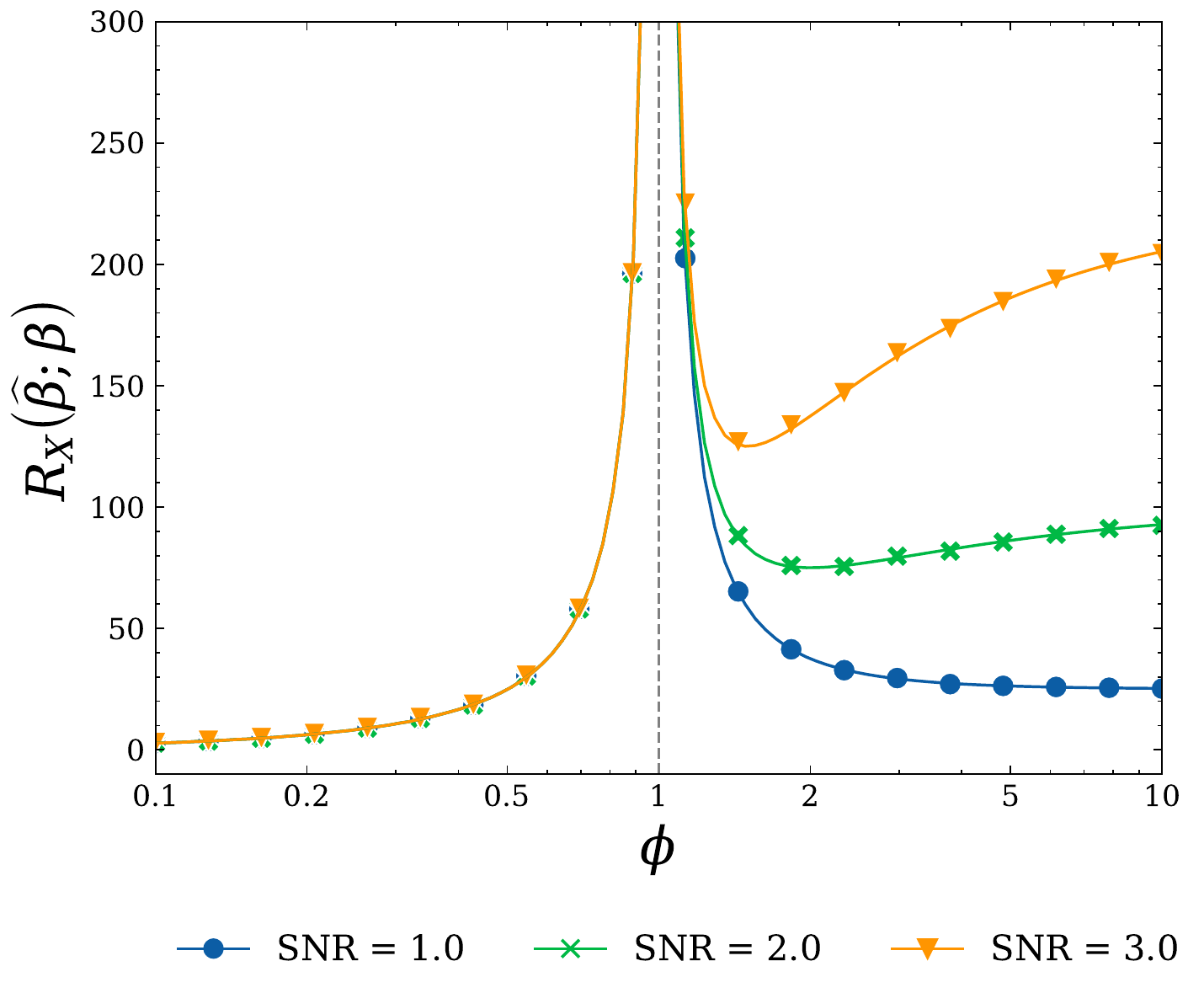}
    \caption{\small 
    Asymptotic risk curves for the ridgeless least square estimator, as functions of $\phi = p/n$. The blue, green, and yellow lines are theoretical risk curves for $\SNR = \alpha / \sigma = 1, 2, 3$ with $(\alpha, \sigma)$ taking $(5, 5), (10, 5)$ and $(15, 5)$, respectively. The dots, crosses, and triangles mark the finite-sample risks with $n=400$, $\phi$ varying in $[0.1,10]$ and $p = [n \phi]$. Each row of the feature matrix $X\in \RR^{n\times p}$ was \iid drawn from $\mathcal{N} (0, I_p)$. }
    \label{fig:fig_0}
\end{figure}

This paper answers the questions above in the case of sketched ridgeless least square estimators \eqref{eq:sketched}, in both the underparameterized and overparameterized regimes, where downsampling is achieved through random sketching. Our intuition is that downsampling  plays a similar role as that of increasing the model capacity. Because increasing the model capacity has been recently observed in modern machine learning  to often help improve the generalization performance \citep{he2016deep, neyshabur2014search, novak2018sensitivity, belkin2018reconciling, nakkiran2021deep}, downsampling may also benefit generalization properties. This ``dual view" can be seen clearly in the case of linear regression, where the out-of-sample prediction risk only depends on the model size and sample size via the quantity $p/n$ \citep{hastie2019surprises}; see Figure \ref{fig:fig_0}. Thus increasing the model size ($p$) has an equivalent impact on the generalization performance as reducing the sample size ($n$).

Motivated by this dual view, we examine the out-of-sample prediction risks of the sketched ridgeless least square estimator in the proportional regime, where the sketching size $m$ is comparable to the sample size $n$ and the dimensionality $p$. We consider a broad class of sketching matrices that satisfy mild assumptions, as described in Assumption \ref{assumption_sketchingMatrix}, which includes several existing sketching matrices as special cases. Our work makes the following key contributions.
\begin{enumerate}
\item First, we provide asymptotically exact formulas for the two out-of-sample prediction risks in the proportional regime. This allows us to reveal the statistical role of downsampling in terms of generalization performance. Perhaps surprisingly, we find that downsampling does not always harm the generalization performance and may even improve it in certain scenarios.

\item Second, we show that orthogonal sketching is optimal among all types of sketching matrices considered in the underparameterized case. In the overparameterized case however, all general sketching matrices are equivalent to each other.

\item Third, we identify the optimal sketching sizes that minimize the out-of-sample prediction risks. The optimally sketched ridgeless least square estimators exhibit universally better risk curves when varying the model size, indicating their improved stability compared with the full-sample estimator.

\item Fourth, we propose a practical procedure to empirically determine the optimal sketching size using an additional validation dataset, which can be relatively small in size.

\item  Fifth, in addition to characterizing the first-order limits, we provide central limit theorems for the risks. Leveraging results from random matrix theory for covariance matrices \citep{zhang2007spectral, el2009concentration, Knowles2017, zheng2015substitution}, we establish almost sure convergence results for the test risks. These results complement the work of \cite{dobriban2018asymptotics}, which focused on the asymptotic limits of expected risks. The expected risk results can be recovered from our findings using the dominated convergence theorem.
\end{enumerate}

\subsection{Related work}

\paragraph{Generalization properties of overparameterized models} 
The generalization properties of overparametrized models have received significant attention in recent years. It all began with the observation that overparameterized neural networks often exhibit benign generalization performance, even without the use of explicit regularization techniques \citep{he2016deep, neyshabur2014search, canziani2016analysis, novak2018sensitivity, zhang2021understanding, bartlett2020benign, liang2020just}. This observation challenges  the conventional statistical wisdom that overfitting  the training data leads to poor accuracy on new examples.  To reconcile this discrepancy, \cite{belkin2018reconciling} introduced the unified ``double descent" performance curve that reconciles  the classical understanding with the modern machine learning practice. This double descent curve subsumes the textbook U-shape bias-variance-tradeoff curve  \citep{hastie2009elements} by demonstrating how increasing model capacity beyond the interpolation threshold can actually lead to improved test errors. Subsequent research has aimed to characterize this double descent phenomenon in various simplified models, including linear models \citep{hastie2019surprises,richards2021asymptotics},  random feature models \citep{mei2022generalization}, and partially optimized two-layer neural network \citep{ba2019generalization}, among others. 

\paragraph{Implicit regularization and minimum $\ell_2$-norm solutions} 
Another line of research focuses on understanding the phenomenon of benign overfitting through implicit regularization  mechanisms in overparameterized models  \citep{neyshabur2014search}. For instance, \cite{gunasekar2018characterizing} and \cite{zhang2021understanding} showed that  gradient descent (GD) converges  to the minimum $\ell_2$-norm solutions in  linear regression problems, which corresponds to the ridgeless least square estimators.  
\cite{hastie2019surprises} characterized  the exact out-of-sample prediction risk for the ridgeless least square estimator in the proportional regime. Minimum $\ell_2$-norm solutions are  also studied for other models, including kernel ridgeless regression \citep{liang2020just}, classification \citep{chatterji2021finite,liang2021interpolating, muthukumar2021classification}, and the random feature model \citep{mei2022generalization}.

\paragraph{Paper overview} 
The rest of this paper proceeds as follows. Section \ref{sec:pre} provides the necessary preliminaries for our analysis. In Section \ref{sec:isotropic}, we investigate the out-of-sample prediction risks under the assumption of isotropic features. Section \ref{sec:correlated} focuses on the case of correlated features.  We present a simple yet practical procedure to determine the optimal sketching size in Section \ref{sec:prac}. In Section \ref{sec:discussion}, we extend our results in several directions. Section \ref{sec:conclusion} provides the conclusions and discussions of the study. The details of some numerical experiments, computational cost comparisons, as well as all proofs, are provided in the appendix.

\section{Preliminaries}\label{sec:pre}

In this section, we provide definitions for two types of random sketching matrices,  introduce two out-of-sample prediction risks to measure  the generalization performance, and present several standing assumptions that are crucial for our analysis.

\subsection{Sketching matrix}
A sketching matrix $S \in \RR^{m\times n}$ is used to construct the sketched dataset $(SX, SY)$, allowing us to perform approximate computations  on the sketched dataset for the interest of the full dataset. Recall $m$ is the sketching size and we shall refer to $m/n$ as the downsampling ratio.  We consider two types of sketching matrices: orthogonal sketching matrices and \iid sketching matrices, defined as follows. 

\begin{definition}[Orthogonal sketching matrix]
An orthogonal sketching matrix $S\in \RR^{m\times n}$ is a partial orthogonal random matrix, i.e., $S$ satisfies the condition $SS^\top = I_m$, where $I_m$ denotes the identity matrix of size $m\times m$.
\end{definition}

\begin{definition}[\iid sketching matrix]\label{def:iid}
An i.i.d. sketching matrix $S$ is a random matrix  whose entries are \iid, each with mean zero, variance $1/n$, and a finite fourth moment.
\end{definition}

For i.i.d. sketching, we consider \iid Gaussian sketching matrices in all of our experiments, although our results hold for general \iid sketching matrices. For orthogonal sketching, we construct an orthogonal sketching matrix based on the subsampled randomized Hadamard transforms \citep{ailon2006approximate}. 
Specifically, we use  $S=BHDP$, where the rows of $B \in \RR^{m\times n}$  are sampled  without replacement from the standard basis of  $\RR^n$, $H \in \RR^{n \times n}$ is  a Hadamard matrix \footnote{The definition of Hadamard matrices can be found, for example, in \citep{ailon2006approximate}.}, $D \in \RR^{n \times n}$ is a diagonal matrix of i.i.d. Rademacher random variables, and $P \in \RR^{n \times n}$ is a uniformly distributed permutation matrix. The time complexity of computing $(SX,SY)$ is of order $\cO(np\log m )$. Orthogonal sketching matrices can also be realized by, for example, subsampling and Haar distributed matrices \citep{mezzadri2006generate}.


\subsection{Out-of-sample prediction risk}

Recall that $\hat\beta^{S}$ is the sketched ridgeless regression estimator defined in Equation \eqref{eq:sketched}. Let us consider a test data point $x_{\text{new}}\sim P_x$, which is independent of the training data. 
Following \cite{hastie2019surprises},  we consider the following out-of-sample prediction risk as a measure of the generalization performance:
\$
R_{(\beta, S,X)} \left(\hat\beta^{S};\beta\right)&=\EE\left[\left(x_{\text{new}}^\top \hat\beta^{S} - x_{\text{new}}^\top  \beta\right)^2  \Big| \beta, S,X \right]
=\EE\left[ \left\|\hat\beta^{S}-\beta\right\|_\Sigma^2 \Big|\beta, S,X \right],
\$
where $\Sigma:= \cov(x_i)$ is the covariance matrix of $x_i$, and $\|x\|_\Sigma^2:=x^\top  \Sigma x$.  The above conditional expectation  is taken with respect to the randomness of $\{\varepsilon_i\}_{1\le i\le n}$ and $x_{\text{new}}$,  while  $\beta, S$, and $X$ are fixed. 
We can decompose the out-of-sample prediction risk into bias and variance components:
\#\label{risk_decomposition}
&R_{(\beta, S,X)}\rbr{\hat{\beta}^S;\beta}
= B_{(\beta, S,X)}\left(\hat\beta^{S};\beta\right)+V_{(\beta,S,X)}\left(\hat\beta^{S}; \beta\right), 
\#
where 
\#
B_{(\beta, S,X)}\left(\hat\beta^{S};\beta\right) &= \left\|\EE\left(\hat\beta^{S}|\beta, S,X\right)-\beta\right\|_\Sigma^2,   \label{bias1}\\ 
V_{(\beta,S,X)}\left(\hat\beta^{S}; \beta\right) &= {\rm tr}\left[{\rm Cov}\left(\hat\beta | \beta,S,X\right)\Sigma\right].  \label{variance1}
\#
We also consider a second out-of-sample prediction risk, defined as: 
\$
R_{(S,X)} \left(\hat\beta^{S};\beta\right)&=\EE\left[\left(x_{\text{new}}^\top \hat\beta^{S} - x_{\text{new}}^\top  \beta\right)^2  \Big|  S,X \right]
=\EE\left[ \left\|\hat\beta^{S}-\beta\right\|_\Sigma^2 \Big| S,X \right].
\$
The second one also averages over the randomness of $\beta$. Similarly, we have the following bias-variance decomposition
\$
R_{(S,X)}\rbr{\hat{\beta}^S;\beta} = B_{( S,X)}\left(\hat\beta^{S};\beta\right)+V_{(S,X)}\left(\hat\beta^{S}; \beta\right),
\$
where
\#
\!\!B_{(S,X)}\!\left(\hat\beta^{S};\beta\right) &\!=\! \EE \sbr{\left\|\EE\left(\hat\beta^{S}|\beta, S,X\right)\!-\!\beta\right\|_\Sigma^2\big| S,\!X},  \label{bias2} \\
\!\!V_{(S,X)}\!\left(\hat\beta^{S}; \beta\right) &\!=\! \EE\cbr{{\rm tr}\!\left[{\rm Cov}\left(\hat\beta | \beta,S,X\right)\!\Sigma\right]\big| S,\!X}.  \label{variance2}
\#

We shall also refer to the above out-of-sample prediction risks as test risks or simply risks, since they are the only risks considered in this paper.  Throughout the paper, we study the above two out-of-sample prediction risks by examining their bias and variance terms respectively.  Specifically,  we study the  behaviors of $R_{(\beta, S,X)}(\hat \beta^{S};\beta)$ and $R_{(S,X)}(\hat \beta^{S};\beta)$ in the proportional asymptotic limit where the sketching size $m$, sample size $n$, and dimensionality $p$ all tend to infinity such that  the aspect ratio converges as  $\phi_n:=p/n \to \phi$, and the downsampling ratio converges as  $\psi_n:=m/n \to \psi \in (0,1)$.
It is worth noting that $R_{(\beta, S,X)}(\hat \beta^{S};\beta)$ exhibits larger variability due to the additional randomness introduced by the random variable $\beta, S, X$ when compared with $R_{(S,X)}(\hat \beta^{S};\beta)$.

\subsection{Assumptions}

This subsection collects  standing assumptions. 

\begin{assumption}[Covariance and moment conditions]\label{assump_model_spec}
    For $i = 1,\cdots,n$, $x_i = \Sigma^{1/2} z_i$, where $z_i$ has \text{i.i.d.} entries with mean zero, variance one and a finite moment of order $4+\eta$ for some $\eta>0$. The noise $\varepsilon$ is independent of $x$, and follows a distribution $P_\varepsilon$  on $\mathbb R$ with mean $\EE(\varepsilon)=0$ and variance ${\rm var}(\varepsilon)=\sigma^2$. 
\end{assumption}

\begin{assumption}[Correlated features]\label{assumption_general}
    The matrix $\Sigma$ is a deterministic positive definite matrix, and there exist constants $C_0, C_1$ such that $0 < C_0 \leq \lambda_{\min}(\Sigma) \leq \lambda_{\min}(\Sigma) \leq C_1$ for all $n$ and $p$. The empirical spectral distribution (ESD) of $\Sigma$  is defined as $F^\Sigma(x) = \frac{1}{p}\sum_{i=1}^p \mathbf{1}_{[\lambda_i(\Sigma),\infty)}(x)$. Assume that as $p \to \infty$, the ESD $F^\Sigma$ converges weakly to a probability measure $H$.  
\end{assumption}

\begin{assumption}[Random $\beta$]\label{assumption_random_effect}
The coefficient vector $\beta \in \RR^p$ is a random vector with i.i.d. entries satisfying $\EE(\beta) = 0$, $\EE\left((\sqrt{p}\beta_i)^2\right) = \alpha^2$, and $\sup_i\EE\left((\sqrt{p}\beta_i)^{4+\eta}\right) < \infty$ for some $\eta>0$.  It is assumed to be independent of the data matrix $X$, the noise $\varepsilon$, and the sketching matrix $S$. 
\end{assumption}

\begin{assumption}[Sketching matrix]\label{assumption_sketchingMatrix}
Let $S \in \RR^{m \times n} $ be a sketching matrix.  Suppose the ESD of $SS^\top $ converges weakly to a probability measure $B$. {Furthermore, there exist constants $\tilde{C}_0, \tilde{C}_1 > 0$ such that  almost surely for all large $n$, it holds that $0 < \tilde{C}_0 \leq \lambda_{\min}\rbr{SS^\top } \leq \lambda_{\max}\rbr{SS^\top } \leq \tilde{C}_1$.}
\end{assumption}

Assumption \ref{assump_model_spec} specifies the covariance matrix for features and moment conditions for both features and errors  \citep{dobriban2018high, hastie2019surprises,li2021asymptotic}. While \cite{dobriban2018asymptotics} requires only a finite fourth moment for $z_i$, which is slightly weaker than our moment condition, this is because they studied the expected risk, which has less randomness compared to our risks.  Assumption \ref{assumption_general} considers correlated features. In this paper, we first focus on the random $\beta$ case as stated in Assumption \ref{assumption_random_effect}, where $\beta$ has i.i.d. elements, allowing for a clear presentation of optimal sketching size results. The assumption of random $\beta$ is commonly adopted in the literature \citep{dobriban2018high,li2021asymptotic}. We also consider the deterministic $\beta$ in Section~\ref{sec:discussion}, where the interaction between $\beta$ and $\Sigma$ needs to be taken into account. 
Assumption \ref{assumption_sketchingMatrix} regarding the sketching matrix is relatively mild. For example, orthogonal sketching matrices naturally satisfy this assumption. According to \cite{bai1998no}, almost surely there are no eigenvalues outside the support of the limiting spectral distribution (LSD) of large-dimensional sample covariance matrices for sufficiently large sample size. Therefore, the i.i.d. sketching matrices in Definition \ref{def:iid} also satisfy this assumption.

\section{A warm-up case: Isotropic features}\label{sec:isotropic}

As a warm-up, we first study the case of isotropic features, specifically when $\Sigma=I_p$, and postpone the investigation of  the  correlated case to  Section \ref{sec:correlated}. Before presenting the limiting behaviors, we establish the relationship between the two out-of-sample prediction risks through the following lemma, which is derived in the general context of correlated features. 

\begin{lemma}\label{lemma_bias_variance}
Under Assumptions \ref{assump_model_spec} and \ref{assumption_random_effect}, the  biases \eqref{bias1} and \eqref{bias2}, as well as the variances \eqref{variance1} and \eqref{variance2}, can be expressed as follows:
\begin{align*}
B_{(\beta,S,X)}\rbr{\hat{\beta}^S;\beta} 
& =\beta^\top \sbr{(X^\top  S^\top  S X)^+ X^\top  S^\top  S  X -I_p}\Sigma 
\sbr{(X^\top  S^\top  S X)^+ X^\top  S^\top  S  X -I_p}\beta,\\
B_{(S,X)}\rbr{\hat{\beta}^S;\beta}
& = \frac{\alpha^2}{p}\trace \cbr{\sbr{I_p - (X^\top  S^\top  S X)^+ X^\top  S^\top  S  X } \Sigma }, \\
V_{(\beta, S,X)}\left(\hat\beta^{S}; \beta\right)
& = V_{( S,X)}\left(\hat\beta^{S}; \beta\right)
=  \trace \sbr{\sigma^2(X^\top  S^\top  SX)^+ X^\top  S^\top  S S^\top  SX (X^\top  S^\top  SX)^+\Sigma}. 
\end{align*}
Furthermore, suppose there exists $C_1$ such that $\lambda_{\max}\rbr{\Sigma} \leq C_1$. Then as $n,p \to \infty$, 
\begin{align}\label{two_risk_equal_limit}
R_{(\beta,S,X)}\rbr{\hat{\beta}^S;\beta} - R_{(S,X)}\rbr{\hat{\beta}^S;\beta} \overset{{\rm a.s.}}{\to} 0.
\end{align}
\end{lemma}

The above lemma establishes the asymptotic equivalence of the two risks when $\beta$ is random, with the variance terms being exactly equal and the bias terms converging asymptotically. Due to this asymptotic equivalence, our primary focus will be on analyzing the risk $R_{(S,X)}(\hat{\beta}^S;\beta)$. However, it should be noted that the second-order inferential results do not align in general, and this discrepancy will be discussed in detail in Section \ref{sec:discussion}.

\subsection{Limiting risks}

We first focus on the case of isotropic features, which enables us to obtain clean expressions for the limiting risks. 
We characterize the limiting risks  with two  types of sketching matrices: orthogonal and \iid sketching matrices, which were introduced earlier. Recall that we consider $m,n,p\to \infty$ such that $p/n \to \phi$ and $m/n \to \psi\in(0,1)$. 

\begin{theorem}\label{thm:isotropic_limiting_risk}
Under Assumptions \ref{assump_model_spec}, \ref{assumption_random_effect}, \ref{assumption_sketchingMatrix}, and $\Sigma = I_p$,  the following results hold. 
\begin{enumerate}
\item[(i)] If $S$ is an orthogonal sketching matrix, then
\begin{align*}
R_{(S,X)}\rbr{\hat{\beta}^S;\beta}
\overset{{\rm a.s.}}{\to}  
\begin{cases}
\dfrac{\sigma^2\phi\psi^{-1}}{1-\phi\psi^{-1}},\quad &\phi \psi^{-1}< 1,\\
\alpha^2 (1-\psi \phi^{-1})+\dfrac{\sigma^2}{\phi\psi^{-1}-1},\quad & \phi \psi^{-1}>1.
\end{cases}
\end{align*}
\item[(ii)] If $S$ is an \text{i.i.d.} sketching matrix, then 
\begin{align*}
R_{(S,X)}\rbr{\hat{\beta}^S;\beta} 
\overset{{\rm a.s.}}{\to}
\begin{cases}
\dfrac{\sigma^2\phi}{1-\phi} + \dfrac{\sigma^2 \phi\psi^{-1}}{1-\phi\psi^{-1}},\quad &\phi \psi^{-1}< 1,\\
\alpha^2 (1-\psi \phi^{-1})+\dfrac{\sigma^2}{\phi\psi^{-1}-1},\quad & \phi \psi^{-1}>1.
\end{cases}
\end{align*}
\end{enumerate}
Moreover, $R_{(\beta, S,X)}(\hat{\beta}^S;\beta)$ with orthogonal sketching and \iid sketching  converge almost surely to the same limits,  respectively.
\end{theorem}

In the above theorem, we have characterized the limiting risks of sketched ridgeless least square estimators with both orthogonal and \iid sketching. The limiting risks are determined by theoretical risk curves in the underparameterized and overparameterized regimes after sketching, where the regimes are described by $\phi\psi^{-1} < 1$ and $\phi\psi^{-1} > 1$, respectively. We shall simply call these two regimes underparametrized and overparameterized regimes respectively.

Interestingly, orthogonal and \iid sketching exhibit different behaviors in the underparameterized regime, while their limiting risks agree in the overparameterized regime. In the underparameterized regime, taking orthogonal sketching  is strictly better than taking \iid sketching in terms of out-of-sample prediction risks. This difference can be attributed to the distortion of the geometry of the least square regression estimator caused by the non-orthogonality in \iid sketching, as pointed out by \cite{dobriban2018asymptotics}, but for a different risk.  Their risk is the expected version of ours.  By using the dominated convergence theorem, Theorem \ref{thm:isotropic_limiting_risk} can recover their results in the underparameterized case.

Moving to the overparameterized case however, both orthogonal and \iid sketching yield identical limiting risks. Specifically, when $\phi\psi^{-1} > 1$,  the bias term $B_{(S,X)}(\hat\beta^{S};\beta) \overset{{\rm a.s.}}{\to}  \alpha^2 (1-\psi \phi^{-1})$ and the variance term $V_{(S,X)}(\hat\beta^{S}; \beta) \overset{{\rm a.s.}}{\to} {\sigma^2}({\phi\psi^{-1}-1})^{-1}$ hold for both types of sketching.

\begin{figure}[t]
    \centering    \includegraphics[width=\columnwidth]{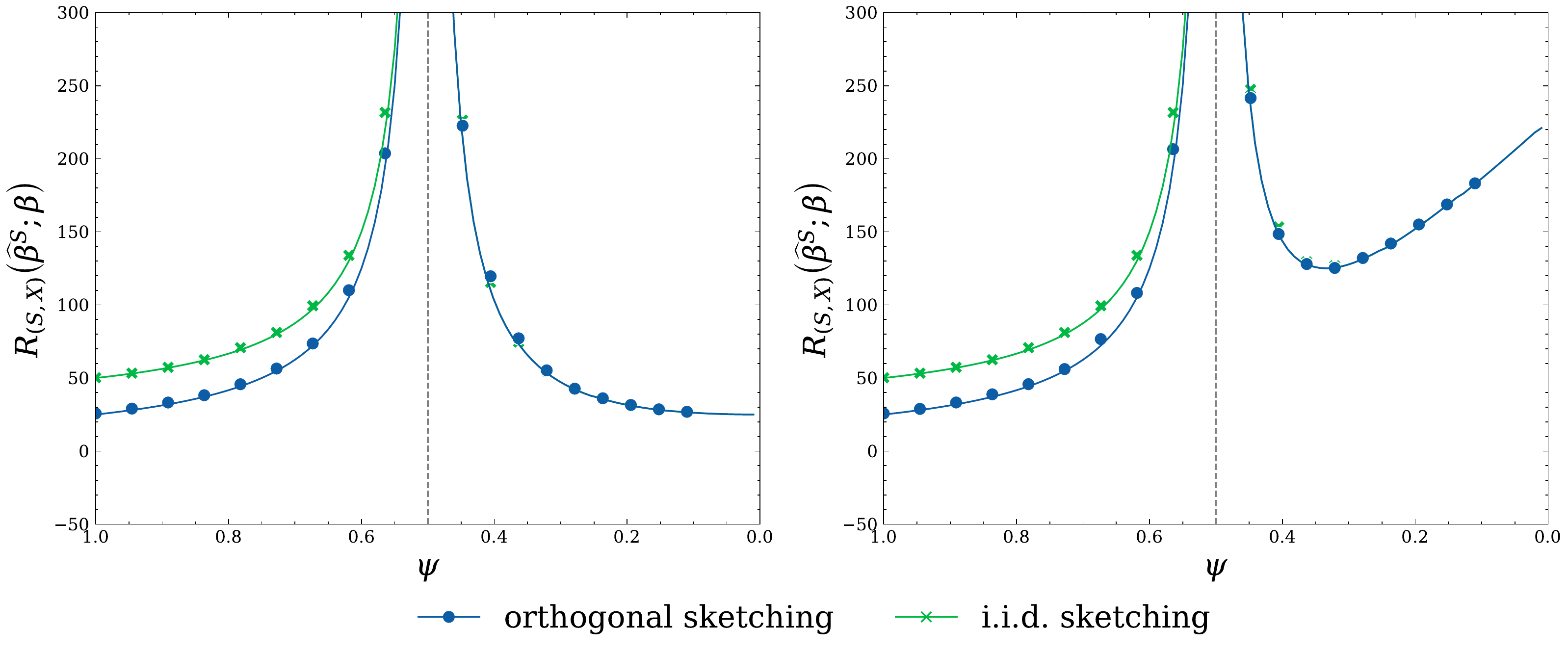}
    \caption{\small 
    Asymptotic risk curves for sketched ridgeless least square estimators with orthogonal and  \iid  sketching under isotropic features, as functions of $\psi$. The  lines in the left panel and right panel are theoretical risk curves for $\SNR = \alpha/\sigma = 1$ with $(\alpha, \sigma)=(5,5)$ and  $\SNR= \alpha/\sigma=3$ with $(\alpha, \sigma)=(15, 5)$, respectively. The blue lines are for orthogonal sketching, while the green lines are for \iid sketching.  The blue dots mark the finite-sample risks for orthogonal sketching, while the green crosses mark the finite-sample risks for \iid sketching, with $n=400$, $p=200$,  $\psi$ varying in $(0,1)$, and $m=[n \psi]$. Each row of  $X\in \RR^{n\times p}$ is \iid drawn from $\mathcal{N}_p (0, I_p)$. 
    The orthogonal sketching matrices are generated using subsampled randomized Hadamard transform, while the entries of the \iid sketching matrices are drawn independently from $\mathcal{N} (0, {1}/{n})$.
    }
    \label{fig:fig_1}
\end{figure}

Let $\SNR= \alpha/\sigma$. Figure \ref{fig:fig_1} plots the asymptotic risk curves as functions of $\psi$,  for sketched ridgeless least square estimators with orthogonal and \iid sketching when $\SNR= \alpha/\sigma = 1, 3$ with  $(\alpha, \sigma) = (5, 5)$ and $(\alpha, \sigma) = (15, 5)$ respectively, along with finite-sample risks. As depicted in the figure, orthogonal sketching is strictly better than \iid sketching in the underparameterized regime, while they are identical  in the overparameterized regime.

Lastly, we compare the limiting risk $R_{(S,X)}(\hat{\beta};\beta)$ of the orthogonally sketched estimator with that of the full-sample estimator,  since orthogonal sketching is universally better than \iid sketching. We can use a variant of \citep[Theorem 1]{hastie2019surprises} to obtain the limiting risk $R_{X}(\hat{\beta};\beta)$ of the full-sample ridgeless least square estimator $\widehat\beta$ with isotropic features:
\$
R_{X}\rbr{\hat{\beta};\beta}  \overset{{\rm a.s.}}{\to}  
\begin{cases}
\dfrac{\sigma^2\phi}{1-\phi},\quad &\phi < 1,\\
    \alpha^2 (1-\phi^{-1})+\dfrac{\sigma^2}{\phi -1},\quad & \phi >1.
\end{cases}
\$
Figure \ref{fig:fig_0} displays the asymptotic risk curves and finite-sample risks of $\hat\beta$. The limiting risk $R_{X}(\hat{\beta};\beta)$ depends on the sample size $n$ and dimensionality $p$ only through the aspect ratio $\phi= \lim p/n$. Comparing this limiting risk with that of the orthogonally sketched estimator in Theorem \ref{thm:isotropic_limiting_risk}, we observe that orthogonal sketching modifies the limiting risk by changing the {\it effective aspect ratio} from $\phi= \lim p/n$ for the original problem to $\phi\psi^{-1} =\lim p/m$ for the sketched problem. This is natural since sketching is a form of downsampling that affects the aspect ratio and, consequently, the limiting risk. Therefore, it is reasonable to ask the following question:
\begin{quote}
\it By carefully choosing the sketching size, can we potentially improve the out-of-sample prediction risks and, consequently, the generalization performance?
\end{quote}
This possibility arises due to the non-monotonicity of the asymptotic risk curves in Figure \ref{fig:fig_0}. In the following subsection, we investigate the optimal sketching size.

\subsection{Optimal sketching size}

In the previous subsection, we discussed the possibility of improving out-of-sample prediction risks and thus generalization performance by carefully choosing the sketching size. We now present the optimal sketching size $m^*$ to minimize the limiting risks for both orthogonal and \iid sketching.

\begin{theorem}[Optimal sketching size for orthogonal and \iid sketching]\label{optimal_sketching_size}
Assume Assumptions \ref{assump_model_spec}, \ref{assumption_random_effect}, \ref{assumption_sketchingMatrix}, and $\Sigma = I_p$.  
The optimal sketching size $m^*$ for both orthogonal and \iid sketching can be determined as follows. 
\begin{itemize}
\item[(a)] If ${\rm SNR}> 1$ and $\phi \in (1-\frac{\sigma}{2\alpha},\frac{\alpha}{\alpha-\sigma}]$,  the optimal sketching size to minimize both limiting risks is $m^*=\frac{\alpha-\sigma}{\alpha}\phi \cdot n$. 
\item[(b)] If  ${\rm SNR}\le 1$ and $\phi \in (\frac{\alpha^2}{\alpha^2+\sigma^2},\infty)$,   taking $\tilde{\beta}=0$ (corresponding to $m^*=0$) yields the optimal solution. 
\item[(c)] 
No sketching is needed if either of the following two holds:
(i) ${\rm SNR}\le 1$ and $\phi \in (0,\frac{\alpha^2}{\alpha^2+\sigma^2}]$, or 
(ii) ${\rm SNR}> 1$ and $\phi \in (0,1-\frac{\sigma}{2\alpha}]\bigcup(\frac{\alpha}{\alpha-\sigma},\infty)$. 
\end{itemize}
\end{theorem}

Theorem \ref{optimal_sketching_size} reveals that both orthogonal and \iid sketching can help improve out-of-sample prediction risks in certain cases. Specifically, when the signal-to-noise ratio is large with $\SNR>1$ and the aspect ratio $\phi$ is within the range $(1-\sigma/(2\alpha), \alpha/(\alpha -\sigma )]$, a nontrivial sketching size of $m^*= (\alpha - \sigma)\phi n/\alpha$ leads to the optimal asymptotic risks. On the other hand, when the signal-to-noise ratio is low and the problem dimension is large, the null estimator $\tilde\beta = 0$, which corresponds to $m^*=0$, is the best among all sketched ridgeless least square estimators.

\begin{figure}[t]
    \centering
 \includegraphics[width=\columnwidth]{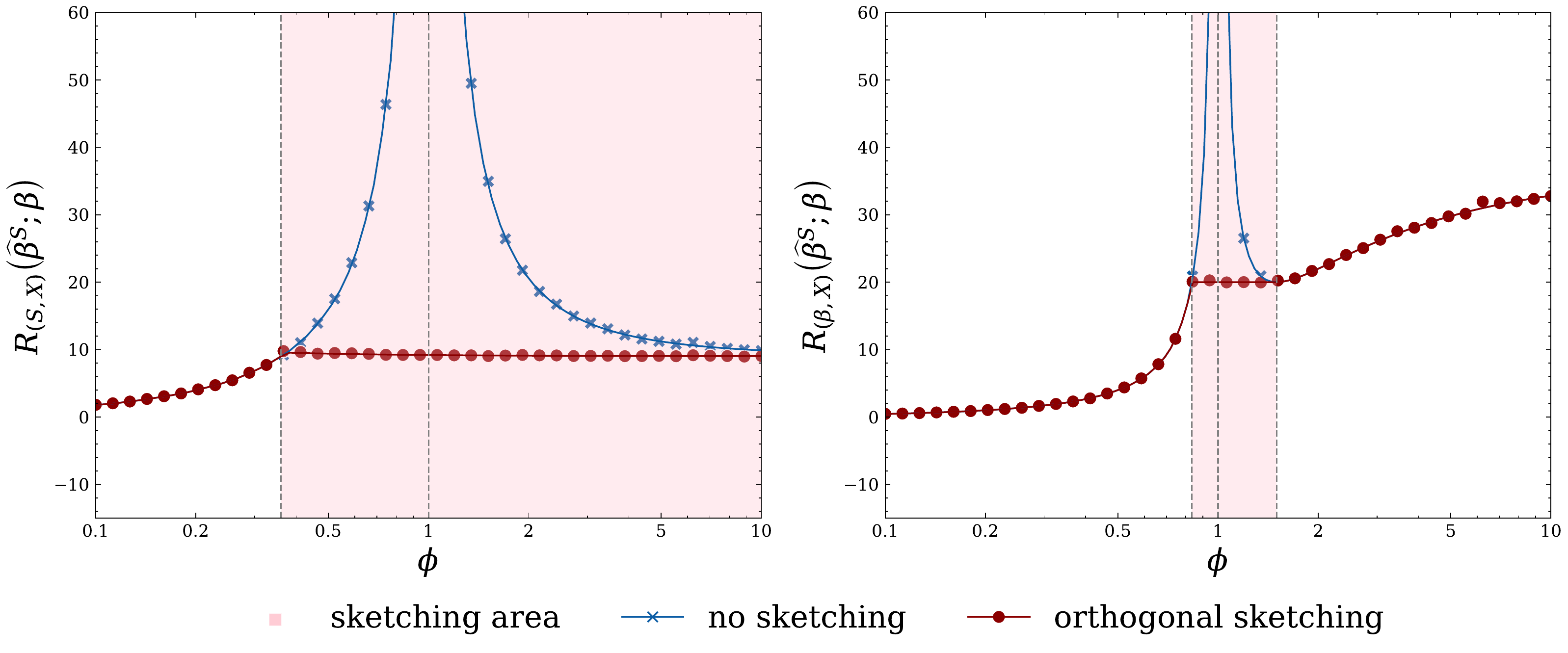}
    \caption{\small 
    Asymptotic risk curves for the full-sample (no sketching) and orthogonally sketched ridgeless least square estimators under isotropic features, as functions of $\phi$. For the sketched estimator, the optimal sketching size $m^*$ is selected based on the SNR and $\phi$, as described in Theorem \ref{optimal_sketching_size}. In the left panel and right panel, the lines  represent the theoretical risk curves for $\SNR = \alpha / \sigma = 0.75$ with $(\alpha, \sigma) = (3, 4)$ and $\SNR = \alpha / \sigma = 3$ with $(\alpha, \sigma) = (6, 2)$, respectively. The blue crosses represent the finite-sample risks for the full-sample estimator, while the red dots indicate the finite-sample risks for the sketched estimator, with  $n=400$, $\phi$ varying in $[0.1, 10]$, and $p=[n \phi]$.   The feature and orthogonal sketching matrices are generated in the same way as in Figure \ref{fig:fig_1}.
    }
    \label{fig:fig_2}
\end{figure}

Figure \ref{fig:fig_2} displays the asymptotic risk curves, as functions of $\phi$,  for the full-sample and optimally sketched ridgeless least square estimators  using orthogonal sketching under isotropic features. As shown in the figure, optimal sketching can stabilize the asymptotic risk curves by eliminating the peaks, indicating that the optimally sketched estimator is a more stable estimator compared to the full-sample one. In Section \ref{sec:prac}, we propose a practical procedure for selecting the optimally sketched estimator.

\section{Correlated features}\label{sec:correlated}

This section considers a general covariance matrix $\Sigma$. 
The results presented here apply to general sketching matrices captured by Assumption \ref{assumption_sketchingMatrix}, including orthogonal and \iid sketching as special cases. We will discuss the overparameterized and underparameterized cases separately.

\subsection{Overparameterized regime}

Recall that $H$ is the limiting spectral distribution (LSD) of $\Sigma$, and $p, m, n \to \infty$ such that $\phi_n=p/n \to \phi$ and $\psi_n =m/n \to \psi \in (0,1)$. In order to analyze the overparameterized case, we need the following lemma.

\begin{lemma}\label{lemma:unique_neg_solution}
    Assume Assumption \ref{assumption_general}.  Suppose  $\phi\psi^{-1}>1$.  Then the following equation \eqref{equation:a} has a unique negative solution with respect to $c_0$,
\begin{align}\label{equation:a}
    1 =  \int \frac{x}{-c_0+x\psi \phi^{-1}} \ dH(x).
\end{align}
\end{lemma}

The above lemma establishes the existence and uniqueness of a negative solution to the equation \eqref{equation:a}. Equations of this type are  known as self-consistent equations \citep{bai2010spectral}, and are fundamental in calculating asymptotic risks. They do not generally admit closed-form solutions but can be solved numerically. To the best of our knowledge, Lemma \ref{lemma:unique_neg_solution}  is not available in the literature. We denote the unique negative solution to \eqref{equation:a} as $c_0 = c_0(\phi, \psi, H)$, which will be used in our subsequent analysis. Our next result characterizes the limiting risks, as well as the limiting biases and variances, in the overparameterized regime. 

\begin{theorem}\label{thm:correlated_over}
    Assume Assumptions \ref{assump_model_spec}-\ref{assumption_sketchingMatrix}. Suppose  
    $\phi \psi^{-1} > 1$. 
    Then the following results hold:
\begin{gather}
B_{(S,X)}(\hat\beta^{S};\beta),\,  B_{(\beta, S,X)}(\hat\beta^{S};\beta) \overset{\as}{\to} -\alpha^2c_0, \label{equation:correlated_over_bias}\\
V_{(S,X)}(\hat\beta^{S};\beta) 
= V_{(\beta, S,X)}(\hat\beta^{S};\beta) 
\overset{\as}{\to}~   \sigma^2\frac{\int \frac{x^2\psi\phi^{-1}}{\rbr{c_0 - x\psi\phi^{-1}}^2}\, dH(x)}{1 -\int \frac{x^2\psi\phi^{-1}}{\rbr{c_0 - x\psi\phi^{-1}}^2}\, dH(x)}. \label{equation:correlated_over_variance}
\end{gather}
Consequently, the limiting risks $R_{(S,X)}(\hat{\beta}^S;\beta)$ and $R_{(\beta,S,X)}(\hat{\beta}^S;\beta)$ converge almost surely to the sum of the right-hand sides of \eqref{equation:correlated_over_bias} and \eqref{equation:correlated_over_variance}.
\end{theorem}

Different from the case with isotropic features, the asymptotic risk in the presence of correlated features does not admit closed-form solutions. However, it can be computed numerically. When $\Sigma = I_p$, the limiting spectral distribution $H$ degenerates to the Dirac measure $\delta_{{1}}$. In this case, we can show  $c_0 = \psi\phi^{-1} - 1$, $B_{(S,X)}(\hat\beta^{S};\beta) \to \alpha^2 (1-\psi \phi^{-1})$,  and $V_{(S,X)}(\hat\beta^{S}; \beta) \to \frac{\sigma^2}{\phi\psi^{-1}-1}$. These results recover  Theorem \ref{thm:isotropic_limiting_risk} for the case of isotropic features.
Furthermore, in the overparameterized regime, the limiting risks do not depend on a specific sketching matrix. This generalizes the same phenomenon observed in Theorem \ref{thm:isotropic_limiting_risk} for isotropic features.

\subsection{Underparameterized regime}

Recall from Assumption~\ref{assumption_sketchingMatrix} that $B$ is the limiting spectral distribution of $SS^\top$. We define $\tilde{c}_0 = \tilde{c}_0(\phi,\psi,B)$ as the unique negative solution to the self-consistent equation:
\begin{align}\label{equation:correlated_under}
1 = \psi \int \frac{x}{-\tilde{c}_0 + x\phi} \ dB(x).
\end{align}
Now we present the results for the limiting risks in the underparameterized regime, as well as the limiting biases and variances.

\begin{theorem}\label{correlated_under_biasvariance}
Assume Assumptions \ref{assump_model_spec}-\ref{assumption_sketchingMatrix}.  Suppose $\phi \psi^{-1} < 1$. Then 
\begin{gather}
B_{(S,X)}(\hat\beta^{S};\beta),\,  B_{(\beta, S,X)}(\hat\beta^{S};\beta) 
\overset{\as}{\to} 0, \label{limit:correlated_under_bias}\\
V_{(S,X)}(\hat\beta^{S};\beta) 
= V_{(\beta, S,X)}(\hat\beta^{S};\beta)
\overset{\as}{\to} \sigma^2\frac{\psi\int \frac{x^2\phi}{\rbr{\tilde{c}_0-x\phi}^2} \, dB(x)}{1-\psi\int \frac{x^2\phi}{\rbr{\tilde{c}_0-x\phi}^2} \, dB(x)}. \label{equation:correlated_under_variance}
\end{gather}
Consequently,  both $R_{(S,X)}(\hat{\beta}^S;\beta)$ and $R_{(\beta,S,X)}(\hat{\beta}^S;\beta)$ converge almost surely to the right hand side of  \eqref{equation:correlated_under_variance}. 
\end{theorem}

\begin{figure}[t]
    \centering
    \includegraphics[width=\columnwidth]{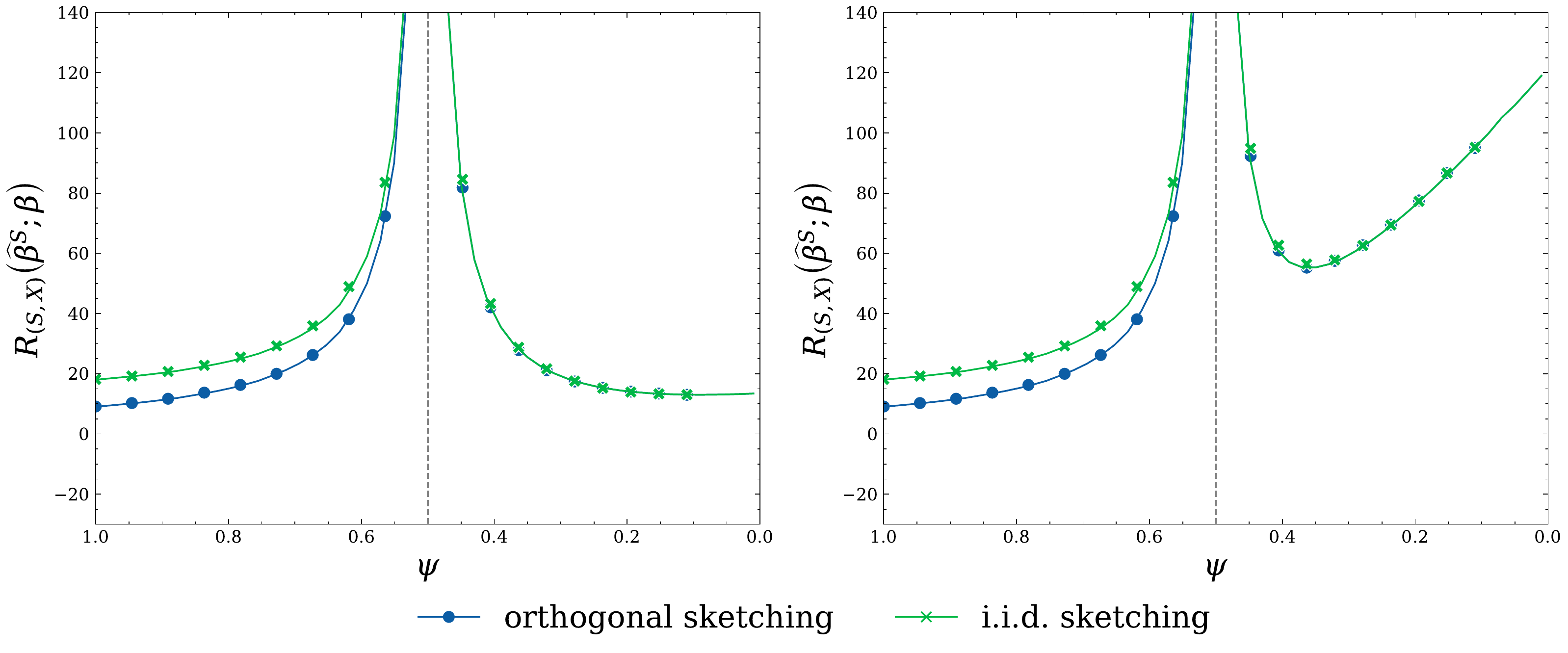}
    \caption{\small 
    Asymptotic risk curves  for  sketched ridgeless least square estimators with orthogonal and \iid sketching under correlated features, as functions of $\psi$. The lines in the left panel and the right panel are theoretical risk curves for $\SNR = \alpha/\sigma = 1$ with $(\alpha, \sigma)= (3,3)$ and $\SNR = \alpha/\sigma = 3$ with $(\alpha, \sigma) = (9, 3)$, respectively. The blue dots mark the finite-sample risks for orthogonal sketching, while the green crosses mark the risks for \iid sketching, with $n=400$, $p=200$, $\psi$ varying in $(0,1)$, and $m=[n \psi]$. Each row of  $X\in \RR^{n\times p}$ is \iid drawn from $\mathcal{N}_p (0, \Sigma)$ and $\Sigma$ has empirical spectral distribution $F^\Sigma(x) = \frac{1}{p}\sum_{i=1}^p \mathbf{1}_{[\lambda_i(\Sigma),\infty)}(x)$ 
    with $\lambda_i = 2$ for $i = 1, \dots, [p/2]$, and $\lambda_i = 1$ for $i = [p/2] + 1, \dots, p$. The orthogonal and \iid sketching matrices are generated in the same way as in Figure \ref{fig:fig_1}. 
    }\label{fig:fig_3}
\end{figure}

In the underparameterized case, the biases vanish, and the variances depend on the sketching matrix $S$ and are independent of the covariance matrix $\Sigma$. The following corollary presents the limiting variances for orthogonal and \text{i.i.d.} sketching.

\begin{corollary}\label{Corollary:under_special_sketching}
    Assume the same assumptions as in Theorem \ref{correlated_under_biasvariance}. The following hold. 
\begin{enumerate}
    \item[(i)] If $S$ is an orthogonal sketching matrix, then
    \begin{align}
        V_{(S,X)}(\hat\beta^{S};\beta) &= V_{(\beta, S,X)}(\hat\beta^{S};\beta) \overset{\as}{\to}\, \sigma^2\frac{\phi\psi^{-1}}{1-\phi\psi^{-1}}.
    \end{align}
    \item[(ii)] If $S$ is an \text{i.i.d.} sketching matrix, then
    \begin{align}
        V_{(S,X)}(\hat\beta^{S};\beta) = V_{(\beta, S,X)}(\hat\beta^{S};\beta) \overset{\as}{\to}\, \sigma^2\rbr{\frac{\phi}{1-\phi}+\frac{\phi\psi^{-1}}{1-\phi\psi^{-1}} }.
    \end{align}
\end{enumerate}
\end{corollary}

The corollary above once again confirms that taking \text{i.i.d.} sketching yields  a larger limiting variance compared to taking orthogonal sketching, extending the corresponding results for isotropic features. This naturally raises the question: 
\begin{quote}
\it Is orthogonal sketching matrix optimal among all sketching matrices?
\end{quote}
We provide a positive answer to this question by utilizing the variance formula \eqref{equation:correlated_under_variance}. Specifically, the following result demonstrates that the Dirac measure, which corresponds to orthogonal sketching, minimizes the variance formula \eqref{equation:correlated_under_variance} and therefore minimizes the limiting risks.

\begin{corollary}[Optimal sketching matrix]\label{coro:optimalsketchingmatrix} 
Taking $B = \delta_{{a}}$ with some $a>0$, which corresponds to orthogonal sketching, minimizes the limiting variance \eqref{equation:correlated_under_variance}, and therefore minimizes the limiting risks, among all choices of $B$ supported on the positive real line $\RR_{>0}$.
\end{corollary}

Figure \ref{fig:fig_3} displays the asymptotic risk curves of sketched ridgeless least square estimators with orthogonal or \iid sketching, under correlated features, as functions of $\psi$. The figure highlights that, when considering a general feature covariance matrix $\Sigma$, employing orthogonal sketching outperforms \iid sketching in the underparameterized regime. However,  both approaches yield identical limiting risks in the overparameterized regime. Furthermore, Figure \ref{fig:fig_4} compares the full-sample and sketched least square estimators. It demonstrates that optimal orthogonal and \iid sketching techniques can enhance the stability of the risk curve by eliminating the peaks observed in the risk curves for the full-sample estimator.

\begin{figure}[t]
    \centering
    \includegraphics[width=\columnwidth]{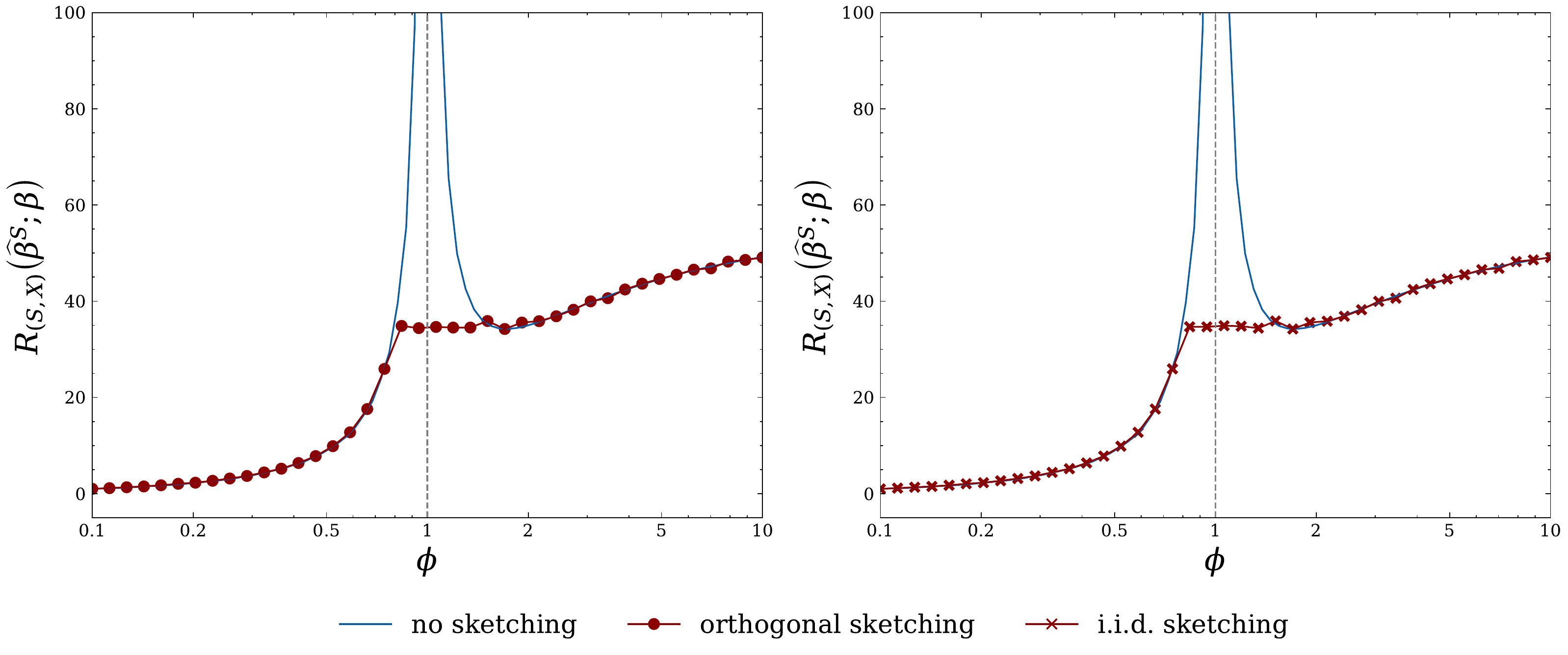}
    \caption{\small 
    Asymptotic risk curves for the full-sample (no sketching) and sketched ridgeless least square estimators with  orthogonal or \iid sketching under correlated features, as functions of $\phi$. For the sketched estimator, the optimal sketching size $m^*$ is selected based on theoretical risk curves, as described in Appendix \ref{appendix:numerical_sims}. The blue lines are the theoretical risk curves for the full-sample estimator with $\SNR = \alpha/\sigma = 2$, where $(\alpha, \sigma) = (6, 3)$. The red dots and crosses mark the finite-sample risks of the orthogonally and \iid sketched estimators, respectively, with  $n = 400$, $\phi$ varying in $[0.1, 10]$, and $p = [n \phi]$. The feature matrix, orthogonal sketching matrices, and \iid sketching matrices are generated in  the same way as in Figure \ref{fig:fig_3}.
    }
    \label{fig:fig_4}
\end{figure}

\begin{figure}[h]
    \centering
    \includegraphics[width=\columnwidth]{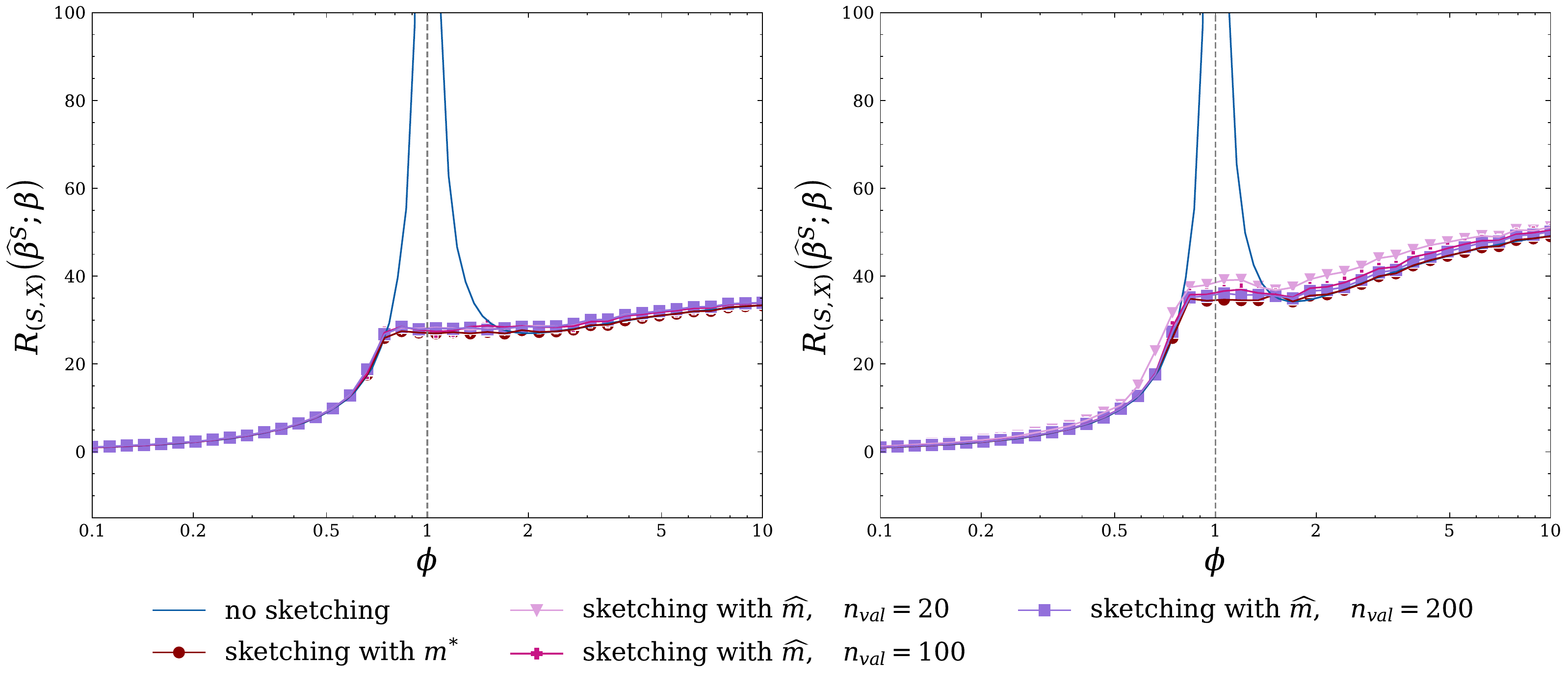}
    \caption{\small 
    Asymptotic risk curves for the full-sample (no sketching) and sketched ridgeless least square estimators with orthogonal sketching under isotropic and correlated features, respectively, as functions of $\phi$. 
    The blue lines in the left panel and the right panel are theoretical risk curves for the full-sample estimator under isotropic features and correlated features, respectively.  For both figures, we set $\SNR = \alpha/\sigma = 2$ with $(\alpha, \sigma) = (6, 3)$. The red dots mark finite-sample risks of the sketched estimator with the theoretically optimal sketching size  ${m}^*$, while the plum triangles, pink diamonds, and  purple squares mark finite-sample risks of sketched estimators with the empirically optimal sketching size $\widehat{m}$ determined  using the validation datasets of sizes $n_{\text{val}} = 20$, $n_{\text{val}} = 100$, $n_{\text{val}} = 200$, where $n = 400$, $\phi$ varies in $[0.1,10]$, and $p=[n \phi]$. The feature matrix and orthogonal sketching matrices are generated in the same way as in Figure \ref{fig:fig_3}. 
    }
    \label{fig:fig_6}
\end{figure}

\section{A practical procedure}\label{sec:prac}

Determining the optimal sketching size based on the theoretical risk curves requires the knowledge of SNR, which is often unknown in practice. Therefore, estimating the optimal sketching size $m^*$ would require new and potentially complicated methodologies for estimating the SNR, which is beyond the scope of this work.

In this section, we present a simple yet practical procedure to pick the best possible sketching size when we have access to an additional validation dataset. This is not very restrictive, especially in applications with large and streaming data where a validation dataset can be easily obtained. Alternatively, we can manually split the dataset into two parts: a training dataset and a validation dataset. The training dataset is used to obtain the sketched estimators, while the validation dataset is used to select the best sketching size. Finally, the test risk $R_{(S,X)}(\hat{\beta}^S;\beta)$ can be evaluated on the testing dataset using the tuned sketched least square estimator.

To evaluate the performance of this procedure, we conducted numerical studies with 500 replications. For each replication, we generated $\beta \sim \mathcal{N}_p(0, \frac{\alpha^2}{p} I_p)$ and created a training dataset $(X, Y)$ with $n = 400$ training samples, a validation dataset $\{(x_{\text{val}, i}, y_{\text{val}, i}): 1\leq i\leq n_{\text{val}}\}$ with  $n_{\text{val}} \in \{20, 100, 200\}$ validation samples, and a testing dataset $\{(x_{\text{new}, i}, y_{\text{new}, i}): 1\leq i \leq n_{\new}\}$ with  $n_{\new} = 100$ testing samples. The feature matrix $X \in \mathbb{R}^{n \times p}$, orthogonal sketching, and \iid sketching matrices were generated in the same way as in Figure \ref{fig:fig_3} and were fixed across all replications.

Next, we provide details on how the sketching size was selected in each replication and how the empirical out-of-sample prediction risks were calculated.

\paragraph{Selection of the optimal sketching size.} The empirically optimal sketching size $\widehat m$ was selected if it minimized the empirical risk across a set of values for $m$ evaluated on the validation dataset. Specifically, given fixed $p$ and $n$, we varied $\psi$ by taking a grid of $\psi \in (0, 1)$ with $|\psi_i - \psi_{i+1}| = \delta$ for  $\delta = 0.05$. This led to a set of potential values for $\widehat m$, i.e., $m_i = [\psi_i n]$.  For each $m_i$, we fitted a sketched ridgeless least square estimator $\widehat{\beta}^{{S_{m_i}}}$ using the training dataset and calculated the empirical risks on the validation dataset: 
\begin{equation}\label{equation:empirical_risk_val}
    \widehat{R}^{\text{val}}_{(S_{m_i}, X)} \left(\widehat{\beta}^{S_{m_i}} ; \beta\right)=\frac{1}{n_{\text {val}}} \sum_{i=1}^{n_{\text {val}}}\left(x_{\text {val}, i}^{\top}\, \widehat{\beta}^{S_{m_i}}-x_{\text {val}, i}^{\top}\,  \beta\right)^2.
\end{equation}
The empirical optimal sketching size $\widehat m$ was picked as the one that minimized the empirical risks across all $m_i$.

We briefly discuss the computational cost of using a validation set. For a given $m$, suppose the computational complexity of orthogonal sketching is  $C(np\log m + m p^2)$ where $C$ is a constant. When $m$ varies with $\abr{m_i-m_{i+1}}=[\delta n]$, the total computational complexity would be $\sum_{i}C(np\log m_i + m_i p^2) \sim C(\frac{1}{\delta}np\log n + \frac{1}{2\delta}np^2)$ where $a_n \sim b_n$ means $\lim_n {a_n}/{b_n}=1$. We compare it with ridge regression which also requires a validation set (or CV) to tune the parameter and should have a computational complexity of $C(\frac{1}{\delta}np^2)$. Although they both have the same order, we can still see sketching reduces almost half of the computational cost and the improvement would be significant especially when $p$ is large. 

\paragraph{Evaluation of the out-of-sample prediction performance.} In the $k$-th replication, we first generate the coefficient vector $\beta(k)$ if the empirically best sketching size was  $\widehat m(k) = n$, we fitted a ridgeless least square estimator $\widehat{\beta}(k)$ on the training set; if $\widehat m(k)  < n$, we fitted a sketched ridgeless least square estimator with the selected $\widehat m (k)$. Denote this final estimator by $\widehat \beta(k)^{S_{\widehat m(k)}}$. The empirical risk of this final estimator was then evaluated on the testing dataset:
\begin{equation}\label{equation:empirical_risk_test}
    \widehat{R}_{(S, X)}\left(\widehat{\beta}^S ; \beta\right)=\frac{1}{500} \sum_{k=1}^{500}\left\{\frac{1}{n_{\text {new}}} \sum_{r=1}^{n_{\text {new}}}\left(x_{\text {new}, r}^{\top} \widehat{\beta}(k)^{S_{\widehat m (k)}}-x_{\text {new}, r}^{\top} \beta(k)\right)^2\right\}.
\end{equation}

Figure \ref{fig:fig_6} plots the asymptotic risk curves for the full-sample and sketched least square estimators with orthogonal sketching, correlated features, and the theoretically and empirically optimal sketching sizes. The performance of the orthogonal sketched estimator with $\widehat{m}$ is comparable to that of sketched estimators with $m^*$ when $n_{\text{val}} = \{20, 100, 200\}$. As the size of the validation dataset increases, the finite-sample risk curve of the orthogonally sketched estimator with $\widehat{m}$ becomes stabler and closer to that of the orthogonally sketched estimator with  $m^*$. Moreover, a particularly small validation dataset with $n_{\text{val}}=20$ already suffices for producing an estimator with a stable and monotone risk curve. 

\section{Extensions}\label{sec:discussion}
\subsection{Deterministic $\beta$ case}

Previously, we assume that the coefficient vector $\beta$ is independent of the data matrix $X$,  and has mean $0$ and covariance  $ p^{-1}{\alpha^2}I_p$ respectively. This section considers deterministic $\beta$  as specified in the following assumption.

\begin{assumption}[Deterministic $\beta$]\label{assumption:nonrandom}
    The coefficient vector $\beta$ is deterministic. 
\end{assumption}


Denote the eigenvalue decomposition of $\Sigma$ by $\Sigma = \sum_{i=1}^p \lambda_iu_i u_i^\top $ where, under  Assumption \ref{assumption_general}, $C_1 \geq \lambda_{1} \geq \lambda_{2} \geq \cdots \geq \lambda_{p} \geq C_0 >0 $. We define the eigenvector empirical spectral distribution (VESD) to be
\begin{align}\label{Def:VESD}
    G_n(x) = \frac{1}{\nbr{\beta}^2}\sum_{i=1}^p \inner{\beta}{u_i}^2\mathbf{1}_{[\lambda_i,\infty)}(x),
\end{align}
where the indicator function $\mathbf{1}_{[\lambda_i,\infty)}(x)$ takes value $1$ if and only if $x\in [\lambda_i,\infty)$. \eqref{Def:VESD} characterizes the relation of $\Sigma$ and $\beta$.  
Theorem \ref{thm:correlated_deterministic} presents the asymptotic risk when $\beta$ is deterministic.  According to Lemma \ref{lemma_bias_variance}, the variance term is exactly the same as in the previous two subsections. Besides,  the bias vanishes in the underparameterized regime.  Thus, the only nontrivial case is the bias for the overparameterized case. Let 
\begin{align}
    c_1:=\frac{\int \frac{ x^2\psi\phi^{-1}}{\rbr{c_0-x\psi\phi^{-1}}^2} \, dH(x)}{1-\int \frac{ x^2\psi\phi^{-1}}{\rbr{c_0-x\psi\phi^{-1}}^2} \, dH(x)},
\end{align}
where $c_0$ is defined in \eqref{equation:a}.
The $c_1$ can be treated as a rescaled limiting variance of the sketched estimator in the overparameterized regime; see \eqref{equation:correlated_over_variance}.

\begin{theorem}\label{thm:correlated_deterministic}
   Assume Assumptions \ref{assump_model_spec}, \ref{assumption_general},  \ref{assumption_sketchingMatrix}, and \ref{assumption:nonrandom}. 
   Then the followings hold. 
   \begin{enumerate}
    \item[(i)] If $p/m \to \phi \psi^{-1} < 1$, 
    \begin{align}
        B_{(\beta, S,X)}(\hat\beta^{S};\beta) \overset{\as}{\to} 0.
    \end{align}
    \item[(ii)] If $p/m \to \phi \psi^{-1} > 1$ 
    and assume the VESD $G_n$ defined in \eqref{Def:VESD} converges weakly to a probability measure $G$, then 
    \begin{align}\label{limit:deterministic_bias2}
        B_{(\beta, S,X)}(\hat\beta^{S};\beta)/\nbr{\beta}^2 \overset{\as}{\to}  \rbr{1+c_1}\int \frac{c_0^2x}{\rbr{c_0-x\psi\phi^{-1}}^2} \, dG(x).
    \end{align}    
    \end{enumerate}
    For the variance term, $V_{(\beta, S,X)}(\hat\beta^{S};\beta)$ converges to the same limit as \eqref{equation:correlated_over_variance} and  \eqref{equation:correlated_under_variance} respectively for the overparameterized and underparameterized cases.
\end{theorem}

\cite{hastie2019surprises} obtained a similar result in the case of the full-sample ridgeless least square estimator.  Because we are dealing with sketched estimators where additional random sketching matrices are involved, our proofs are more challenging. Specifically, we utilize results for separable covariance matrices.  If we further assume $\nbr{\beta}^2 \to \alpha^2$ and $\Sigma = I_p$, then Theorem \ref{thm:correlated_deterministic} shall recover the same limiting risks  in Theorem \ref{thm:isotropic_limiting_risk}.

\subsection{Central limit theorem}\label{sec:clt}

This subsection establishes central limit theorems for both out-of-sample prediction risks $R_{(\beta,S,X)}(\hat \beta^S;\beta)$ and $R_{(S,X)}(\hat \beta^S;\beta)$. \cite{li2021asymptotic} studies the central limit theorems for risks of full-sample ridgeless least square estimator. Compared with their work, our results show the risks of sketched estimators may have smaller asymptotic variances. We start with the following assumptions. 

\begin{assumption}[Random $\beta$]\label{assumption_random_effect_normal}
The coefficient vector $\beta$ follows a multivariate normal distribution $\mathcal{N}_p\rbr{0, \frac{\alpha^2}{p}I_p}$, and is independent of the data matrix $X$, the noise $\varepsilon$, and the  sketching matrix $S$. 
\end{assumption}

\begin{assumption}
\label{CLT_assumption}
Suppose $\{X_{ij}\}$ share the fourth moment $\nu_4:=\EE | X_{ij}|^4<\infty $. Furthermore, they satisfy the following Lindeberg condition  
\$
\frac{1}{np}\sum_{1\le i\le n, \ 1\le j\le p}\EE \left(|X_{ij}|^4 \mathbf{1}_{[\eta\sqrt n,\infty)}(|X_{ij}|)\right)\rightarrow 0,\quad {\rm for\ any\ fixed}\ \eta>0,
\$
where the indicator function $\mathbf{1}_{[\eta\sqrt n,\infty)}(|X_{ij}|)$ takes value $1$ if and only if $|X_{ij}|\ge \eta \sqrt n$.
\end{assumption}

\paragraph{CLTs for $R_{(S,X)}(\hat \beta^S;\beta)$.}
The following theorems give CLTs for $R_{(S,X)}(\hat \beta^S;\beta)$ in the underparameterized  and overparameterized regimes. Recall that $m,n,p \to \infty$ such that $\phi_n = p/n \to \phi$ and $\psi_n = m/n \to \psi \in (0,1)$. 

\begin{theorem}\label{CLT_under}
Assume Assumptions \ref{assump_model_spec},  \ref{assumption_general}, \ref{assumption_sketchingMatrix}, \ref{assumption_random_effect_normal}, and \ref{CLT_assumption}. Suppose  $\phi \psi^{-1} < 1$ and $S$ is an orthogonal sketching matrix. Then it holds that  
\$
p\left(R_{(S,X)}(\hat \beta^S;\beta)- \frac{\sigma^2 \phi_n \psi_n^{-1}}{1-\phi_n\psi_n^{-1}} \right)
\overset{D}{\longrightarrow} \cN (\mu_1, \sigma^2_1),
\$
where 
\$
\mu_1 & = \frac{\sigma^2\phi^2\psi^{-2}}{(\phi\psi^{-1}-1)^2}+\frac{\sigma^2\phi^2\psi^{-2}(\nu_4-3)}{1-\phi\psi^{-1}},\quad
\sigma_1^2  =\frac{2\sigma^4\phi^3\psi^{-3}}{(\phi\psi^{-1}-1)^4}+\frac{\sigma^4\phi^3\psi^{-3}(\nu_4-3)}{(1-\phi\psi^{-1})^2}.
\$
\end{theorem}

\begin{theorem}\label{CLT_over}
Assume Assumptions \ref{assump_model_spec}, \ref{assumption_random_effect_normal}, \ref{CLT_assumption} and $\Sigma = I_p$. Suppose  $ \phi \psi^{-1} > 1$ and $S$ is any sketching matrix that satisfies Assumption \ref{assumption_sketchingMatrix}. Then it holds that 
\$
p\left(R_{(S,X)}(\hat \beta^S;\beta)- \alpha^2 (1-\psi_n \phi_n^{-1})-\frac{\sigma^2}{\phi_n\psi_n^{-1}-1}\right)
\overset{D}{\longrightarrow} \cN (\mu_2, \sigma^2_2),
\$
where 
\$
\mu_2 & = \frac{\sigma^2 \phi\psi^{-1}}{(\phi\psi^{-1}-1)^2}+\frac{\sigma^2 (\nu_4-3)}{\phi\psi^{-1}-1},\quad
\sigma_2^2  =\frac{2\sigma^4\phi^3\psi^{-3}}{(\phi\psi^{-1}-1)^4}+\frac{\sigma^4\phi\psi^{-1}(\nu_4-3)}{(\phi\psi^{-1}-1)^2}.
\$
\end{theorem}

The CLT of $R_{(S,X)}(\hat \beta^S;\beta)$ after an orthogonal sketching $(X,Y)\mapsto (SX, SY)$ coincides with that by \cite{li2021asymptotic} after replacing $p/n$ by $p/m$. 
According to Theorems~\ref{CLT_under}, \ref{CLT_over} and \ref{optimal_sketching_size}, we provide the asymptotic variance of $R_{(S,X)}(\hat \beta^S;\beta)$ for  the orthogonal sketched estimator with  the optimal sketching size $m^*$ given by Theorem~\ref{optimal_sketching_size}.  

\begin{corollary}\label{CLT_optimal}
Denote the asymptotic variance of the risk $R_{(S,X)}(\hat \beta^S;\beta)$ for  the orthogonal sketched estimator with  the optimal sketching size $m^*$ by $\sigma^2_S$. 
The followings hold. 
\begin{itemize}
\item[(a)] If ${\rm SNR}> 1$ and $\phi \in (1-\frac{\sigma}{2\alpha},\frac{\alpha}{\alpha-\sigma}]$, 
then $\sigma_S^2=2\alpha^3 (\alpha -\sigma)+\sigma^2 (\nu_4-3)\alpha (\alpha -\sigma)$. 
\item[(b)] If ${\rm SNR}\le 1$ and $\phi \in (\frac{\alpha^2}{\alpha^2+\sigma^2},\infty)$, then $\sigma_S^2= O(\frac{m^*}{n})\rightarrow 0$.
\item[(c)] If either of the following two holds:
(i) ${\rm SNR}\le 1$ and $\phi \in (0,\frac{\alpha^2}{\alpha^2+\sigma^2}]$, or 
(ii) ${\rm SNR}> 1$ and $\phi \in (0,1-\frac{\sigma}{2\alpha}]\bigcup(\frac{\alpha}{\alpha-\sigma},\infty)$, then 
\$
\sigma_S^2=
\begin{cases}
\dfrac{2\sigma^4\phi^3}{(\phi-1)^4}+\dfrac{\sigma^4\phi^3(\nu_4-3)}{(1-\phi)^2},\ &{\rm if}\ \phi<1,\\
\dfrac{2\sigma^4\phi^5}{(\phi-1)^4}+\dfrac{\sigma^4\phi^3(\nu_4-3)}{(\phi-1)^2},\ &{\rm if}\ \phi> 1.
\end{cases}
\$
\end{itemize}
\end{corollary}

Comparing the asymptotic variance $\sigma_S^2$ of $R_{(S,X)}(\hat \beta^S;\beta)$ with optimal sketching and that of $R_{S, X}(\hat \beta;\beta)$ without sketching, we have following observations. First, the non-trivial and optimal sketching in case (a) may result in a smaller  asymptotic variance $\sigma_S^2$ than  
that for the full-sample estimator. Take standard Gaussian features with $\phi>1$, for which the forth (central) moment $\nu_4$ is 3, as an example. Then it can be verified that $\sigma_S^2\le 2\sigma^4\phi^3(\phi-1)^{-4}=: \sigma^2_0$  for $\phi\in(1,\alpha(\alpha-\sigma)^{-1})$ and $\sigma_S^2\le \sigma^4 (2\phi-1)(1-\phi)^{-4}/8 <\sigma^2_0$ for $\phi\in (1-\sigma\alpha^{-1}/2,1]$ when ${\rm SNR}> 1$. Second, the trivial sketching in case (b) has a zero limiting variance because in this case  the null estimator $\tilde{\beta} = 0$ is optimal.

\paragraph{CLTs for $R_{(\beta, S,X)}(\hat \beta^S;\beta)$.}

In the underparameterized regime, for sufficiently large $n$, $B_{(\beta,S,X)}\rbr{\hat{\beta}^{ S};\beta}\sim B_{(S,X)}\rbr{\hat{\beta}^{ S};\beta} \sim 0$, and $V_{(\beta,S,X)}\rbr{\hat{\beta}^{ S};\beta} \sim V_{(S,X)}\rbr{\hat{\beta}^{ S};\beta}$. Thus, $B_{(\beta,S,X)}\rbr{\hat{\beta}^{ S};\beta}$ has exactly the same CLT as $B_{(S,X)}\rbr{\hat{\beta}^{ S};\beta}$ in Theorem \ref{CLT_under}. We now present the corresponding CLT in the overparameterized case.

\begin{theorem}\label{thm:clt_over_conditioningbeta}
    Assume Assumptions \ref{assump_model_spec}, \ref{assumption_random_effect_normal}, \ref{CLT_assumption} and $\Sigma=I_p$. Suppose $ \phi \psi^{-1} > 1$ and $S$ is any sketching matrix $S$ that satisfies Assumption \ref{assumption_sketchingMatrix}. Then it holds that 
\begin{align*}
    \sqrt{p}\left(R_{(\beta, S,X)}(\hat \beta^S;\beta)- \alpha^2 (1-\psi_n \phi_n^{-1})-\frac{\sigma^2}{\phi_n\psi_n^{-1}-1}\right)
\overset{D}{\longrightarrow} \cN (\mu_3, \sigma^2_3),
\end{align*}
where $\mu_3 = 0$ and $\sigma_3^2 =2(1-\phi^{-1}\psi)\alpha^4$. 
More precise versions of $\mu_3$ and $\sigma_3^2$ are 
\begin{align*}
    \tilde{\mu}_3 &= \frac{1}{\sqrt{p}} \rbr{\frac{\sigma^2 \phi\psi^{-1}}{(\phi\psi^{-1}-1)^2}+\frac{\sigma^2 (\nu_4-3)}{\phi\psi^{-1}-1}},\\
\tilde{\sigma}_3^2  &=2(1-\phi^{-1}\psi)\alpha^4 + \frac{1}{p} \rbr{\frac{2\sigma^4\phi^3\psi^{-3}}{(\phi\psi^{-1}-1)^4}+\frac{\sigma^4\phi\psi^{-1}(\nu_4-3)}{(\phi\psi^{-1}-1)^2}}.
\end{align*}
\end{theorem}

\subsection{Misspecified model}
This subsection briefly discusses the misspecified model. When the misspecification error, aka model bias,  is included,  the risk will decrease at first and then increase for the full-sample ridgeless least square estimator in the underparameterized case. 
This aligns with the classic statistical idea of ``underfitting'' and ``overfitting''. This subsection studies the effect of sketching on the selection of the optimal sketching size. 

We consider a misspecified in which  we observe only a subset of the features. A similar model is also discussed in the section 5.1 of \cite{hastie2019surprises}. Suppose the true model is 
\begin{align}\label{model_misspecified}
y_i=\beta^\top  x_i+ \theta^\top  w_i+ \varepsilon_i,\ i = 1,\cdots,n,    
\end{align}
where $x_i \in \RR^p, w_i \in \RR^q$ and the noise $\varepsilon_i$ is independent of $(x_i,w_i)$. Further assume $(x_i, w_i)$ are jointly Gaussian with mean zero and covariance matrix
\begin{align*}
    \Sigma = 
    \begin{bmatrix}
        \Sigma_{xx}, \Sigma_{xw}\\
        \Sigma_{xw}^\top , \Sigma_{ww}
    \end{bmatrix}.
\end{align*}

We can only observe the data matrix $X = (x_1,\cdots,x_n)^\top  \in\mathbb R^{n\times p}$. Still, we use the sketched data $\tilde Y := S Y\in \RR^{m}$, $ \tilde {X}:= SX\in\RR^{m\times p}$ and its corresponding minimum-norm least square estimator $\hat\beta^{S}$ defined in \eqref{eq:sketched}. Let $(x_{\text{new}},w_{\text{new}})$ be a test point.  The out-of-sample prediction risk is defined as
\begin{align*}
    R_{(S,X)} \left(\hat\beta^{S};\beta, \theta\right)=\EE\left[\left(x_{\text{new}}^\top \hat\beta^{S} - x_{\text{new}}^\top  \beta-w_{\text{new}}^\top  \theta\right)^2  \Big|  S,X \right].
\end{align*}
Here we let $\beta$ and $\theta$ are nonrandom parameters and the expectation is taken over $x_{\text{new}}, w_{\text{new}}, \epsilon$ and also $W = (w_1,\cdots,w_n)^\top  \in\mathbb R^{n\times q}$. Similar to lemma 2 in \cite{hastie2019surprises}, we can decompose the risk into two terms,
\begin{align*}
    R_{(S,X)} \left(\hat\beta^{S};\beta, \theta\right)= \underbrace{\EE\left[\left(x_{\text{new}}^\top \hat\beta^{S} - \EE \rbr{y_{\text{new}}|x_{\text{new}}}\right)^2  \Big|  S,X \right]}_{R^*_{(S,X)} \left(\hat\beta^{S};\beta, \theta\right)} + \underbrace{\EE \sbr{\rbr{\EE \rbr{y_{\text{new}}|x_{\text{new}}} - \EE \rbr{y_{\text{new}}|x_{\text{new}},w_{\text{new}}}}^2}}_{M(\beta,\theta)},
\end{align*}
where $M(\beta,\theta)$ can be seen as the misspecification bias. Notice that conditioning on $x_i$, model \eqref{model_misspecified} is equivalent to $y_i = \tilde{\beta}^\top  x_i + \tilde{\varepsilon_i}$ where $\tilde{\beta} = \beta + \Sigma_{xx}^{-1}\Sigma_{xw}\theta$ and $\tilde{\varepsilon_i} \sim N(0, \tilde{\sigma}^2)$ is independent of $x_i$, $\tilde{\sigma}^2 = \sigma^2 + \theta^\top  \Sigma_{w|x}\theta$. Here $\Sigma_{w|x}$ is the covariance matrix of $w_i$ given $x_i$, i.e., $\Sigma_{w|x} = \Sigma_{ww} - \Sigma_{xw}^\top  \Sigma_{xx}^{-1}\Sigma_{xw}$. Moreover, simple calculation shows $M(\beta,\theta) = \theta^\top  \Sigma_{w|x}\theta$. We refer readers to Remark 2 in \cite{hastie2019surprises} for more details.

We conclude that even for this misspecified model, since $M(\beta,\theta)$ is independent of the sketching matrix $S$, and $R^*_{(S,X)} \left(\hat\beta^{S};\beta, \theta\right)$ can still be approximated using Theorem \ref{thm:correlated_deterministic}, random sketching cannot improve the limiting risks by sketching the estimator to the underparameterized regime. We expect in more complicated models, for example, the random feature model in \cite{mei2022generalization}, sketching to the underparameterized regime might help reduce the limiting risks. We leave this problem to the future.

\section{Conclusions and Discussions}\label{sec:conclusion}

This paper introduces a dual view of overparametrization suggesting that downsampling may also help improve generalization performance. Motiviated by this insight, we investigates the statistical roles of downsampling through random sketching in linear regression estimators, uncovering several intriguing phenomena.  First, contrary to conventional beliefs, our findings demonstrate that downsampling does not always harm the generalization performance. In fact, it can be beneficial in certain cases, challenging the prevailing notion. Second, we establish that orthogonal sketching is optimal among all types of sketching considered in the underparameterized regime. In the overparameterized regime however, all general sketching matrices are equivalent. Third, we provide central limit theorems for the risks and discuss the implications of our results for misspecified models. Lastly, we identify the optimal sketching sizes that minimize the out-of-sample prediction risks under isotropic features. 
The optimally sketched ridgeless least square estimators exhibit universally better risk curves, indicating their improved stability compared with the full-sample estimator.  

We point out that the benefit of optimal sketching arises from the non-monotonic nature of the risk function with respect to the aspect ratio. Interestingly, recent studies \citep{hastie2019surprises} have observed that this non-monotonicity disappears when optimally-tuned ridge regularization is applied. The motivation behind investigating minimum norm estimators, including ridgeless linear regression estimators, stems from the surprising behavior of deep neural networks. Despite lacking explicit regularizers like weight decay or data augmentation, deep neural networks often exhibit a minimal gap between training and test performance \citep{zhang2021understanding}. The ridgeless least square estimator closely mimics the practice in neural networks, making it an intriguing subject for analysis in the context of linear regression.

Furthermore, comparing with downsampling, the optimally-tuned ridge regression is usually more computationally intensive, as there is no computational reduction from downsampling. Downsampling can provide a potential tool for mitigating  the risk with less computational cost. Additionally, we demonstrate that, surprisingly, in certain cases, the sketched ridgeless estimator can have a smaller asymptotic variance compared to the full-sample estimator.
This is unclear for ridge regression.  

As future research directions, it would be interesting to compare the statistical behaviors of ridge  and downsampled estimators, as their comparative properties remain unclear. From a broader perspective, viewing downsampling as a form of regularization raises the question of which regularization approach is optimal among all possibilities. Additionally, we hypothesize that Assumption \ref{assumption_sketchingMatrix} on the sketching matrix can be further relaxed to accommodate cases such as subsampling with replacement, where the limiting spectral distribution of $SS^\top$ contains zero as a mass. Refining the proposed practical procedure with provable guarantees and establishing central limit theorems for more general cases, such as i.i.d. sketching and correlated features, are also promising directions for future exploration.

\section*{Acknowledgements}
Yicheng Zeng is partially supported by the Shenzhen Outstanding Scientific and Technological Innovation Talents PhD Startup Project (Grant RCBS20221008093336086) and by the Internal Project Fund from Shenzhen Research Institute of Big Data (Grant J00220230012). Siyue Yang and Qiang Sun are partially supported by Natural Sciences and Engineering Research Council of Canada (Grant RGPIN-2018-06484) and a Data Sciences Institute Catalyst Grant.

\bibliographystyle{icml2023}
\bibliography{ref}

\newpage
\appendix 

\section*{Appendix}

\paragraph{Overview} 
The details of our numerical studies are included in Appendix \ref{appendix:sec:num}. We compare the computational cost between the sketched and full-sample estimators in Appendix \ref{appendix:sec:computation}. We provide the proofs for results under isotropic features in Appendix \ref{appendix:sec:iso} and the proofs for results under correlated features in Appendix \ref{appendix:proofs_risk}. The proof of Theorem \ref{thm:correlated_deterministic} is provided in Appendix \ref{appendix:proof_nonrandom}, and the proofs for the results on central limit theorems are presented in Appendix \ref{appendix:sec:clt}.

Throughout the appendix, we use $\norm{\cdot}_2$ for the spectral norm of a matrix and use $\norm{\cdot}$ for the $\ell_2$ norm of a vector.

\section{Details on numerical studies}\label{appendix:sec:num}

\subsection{Numerical studies for isotropic features}\label{appendix:numerical_sims_isotropic} 

This section provides additional details on 
the numerical studies for isotropic features to replicate Figures \ref{fig:fig_1}   and  \ref{fig:fig_2}. 

\subsubsection{Figure \ref{fig:fig_1}} \label{appendix:numerical_fig1} 

For Figure \ref{fig:fig_1}, numerical simulations were run $500$ replications. For each replication, we generated $\beta \sim \mathcal{N}_p\left(0, \frac{\alpha^2}{p} I_p\right)$ and a training dataset $(X, Y)$ with $n=400$ training samples, and a testing dataset $\left\{\left(x_{\text {new}, i}, y_{\text{new}, i}\right): 1 \leq i \leq n_{\text {new}}\right\}$ with $n_{\text {new}}=100$ testing samples. The feature, orthogonal sketching, and \iid sketching matrices were generated first and then fixed across all replications. The orthogonal sketching matrix was generated using subsampled randomized Hadamard transform, which relies on the fast Fourier transform. This approach is commonly regarded as a rapid and reliable method for implementing sketching algorithms \citep{dobriban2018asymptotics}. The feature matrix and \iid sketching matrices were generated using Python library \texttt{NumPy}. Other details are given in the caption of Figure \ref{fig:fig_1}. Our implementation is available at \url{https://github.com/statsle/SRLR_python}. 

The finite-sample risks, aka the dots and crosses in Figure \ref{fig:fig_1}, were calculated as functions of $\psi$. Specifically, given $n$, we varied $\psi$ by taking a grid of $\psi \in (0, 1)$ with $|\psi_i - \psi_{i+1}| = \delta$ for $\delta = 0.05$. This led to a grid of values for $m$, i.e., $m_i = [\psi_i n]$. For each replication $k$, we first randomly generated a coefficient vector $\beta(k)$. Within replication $k$ and for each $m_i$, we fitted a sketched ridgeless least square estimator $\widehat{\beta}(k)^{{S}_{m_i}}$ using the training dataset and calculated the empirical risks on the testing dataset:
\begin{equation}\label{equation:equation_A1_empirical}
\widehat{R}_{({S}_{m_i}, X)}\left(\widehat{\beta}^{{S}_{m_i}} ; \beta\right)=\frac{1}{500} \sum_{k=1}^{500}\left\{\frac{1}{n_{\text {new}}} \sum_{r=1}^{n_{\text {new}}}\left(x_{\text {new}_r}^{\top} \widehat{\beta}(k)^{{S}_{m_i}}-x_{\text {new}_r}^{\top} \beta(k)\right)^2\right\}.
\end{equation}

\subsubsection{Figure \ref{fig:fig_2}} \label{appendix:numerical_fig2}

The finite-sample risks, aka the dots in Figure \ref{fig:fig_2}, were calculated as functions of $\phi$. Numerical simulation procedure and data generation followed Section \ref{appendix:numerical_fig1}. The optimal sketching size $m^*$ was selected based on Theorem \ref{optimal_sketching_size}. If $m^* = n$, we fitted a ridgeless least square estimator $\widehat{\beta}$ on the training set; if $m^* < n$, we fitted a sketched estimator $\widehat{\beta}^S$ with $m^*$. The empirical risks $\widehat{R}_{({S}, X)}\left(\widehat{\beta}^{{S}} ; \beta\right)$ were evaluated on the testing dataset in a similar way as in Equation \eqref{equation:equation_A1_empirical}. 
To indicate how SNR and $\phi$ influence the selection of $m^*$, the left panel of Figure \ref{fig:fig_2} presents risks for $\SNR < 1$, and the right panel presents risks for $\SNR > 1$. 

\subsection{Numerical studies for correlated features}\label{appendix:numerical_sims}

This section provides additional details on numerical studies for correlated features  to replicate Figures \ref{fig:fig_3}
and  \ref{fig:fig_4}. 

\subsubsection{Figure \ref{fig:fig_3}} \label{appendix:numerical_fig3}
The numerical simulation procedure generally followed Section \ref{appendix:numerical_fig1}. Instead of isotropic features, we generated correlated features. Other details are given in the caption of Figure \ref{fig:fig_3}. 

\subsubsection{Figure \ref{fig:fig_4}} \label{appendix:numerical_fig4}

\begin{figure}[t!]
    \centering
    \includegraphics[width=\columnwidth]{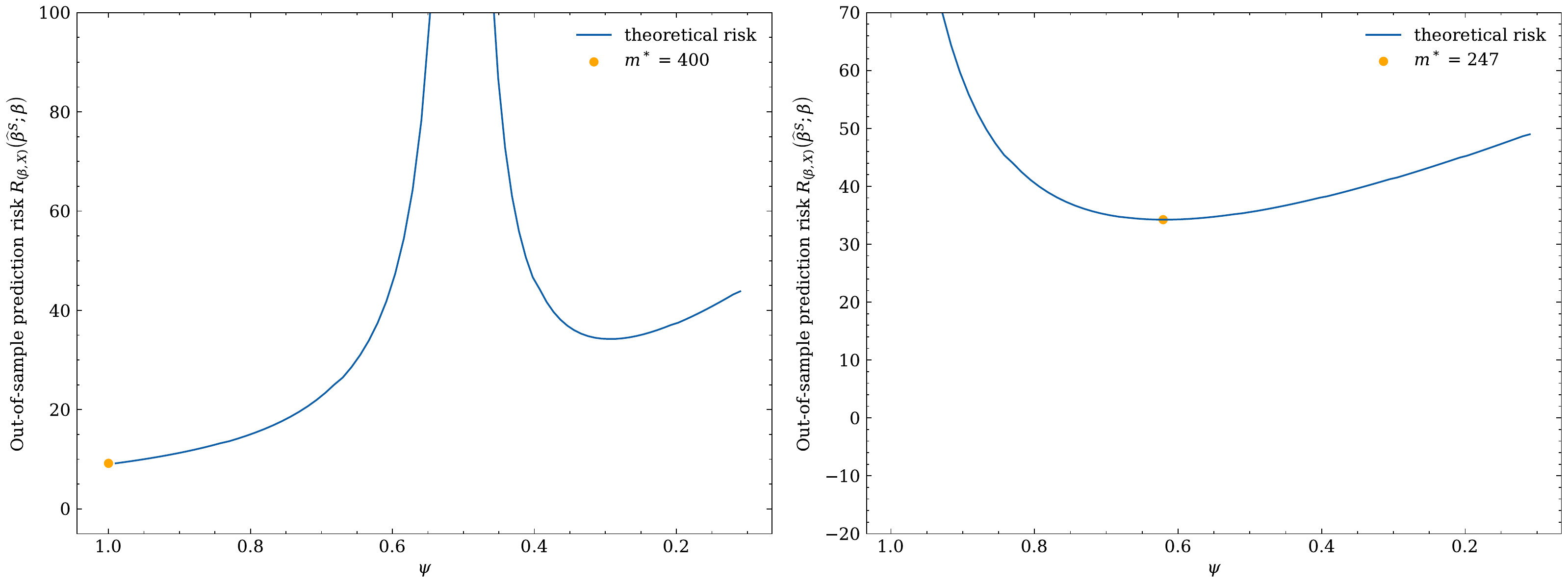}
    \caption{{ Asymptotic risk curves for  sketched ridgeless least square estimators with correlated features, orthogonal sketching, as functions of $\psi$. The lines are theoretical risk curves for $n = 400$ with $p = 200$ and $p = 424$ respectively, where $\SNR = \alpha/\sigma = 2$ with $(\alpha, \sigma)= (6, 3)$,  $\psi$ varying in $(0,1)$, and $m=[n \psi]$. The dot marks the minimum of a theoretical risk curve within $\psi \in (0, 1)$.} }
    \label{fig:fig_5}
\end{figure}

The simulation procedure followed Section \ref{appendix:numerical_fig2} and the data generation followed Section \ref{appendix:numerical_fig3}. 
In the case of correlated  features, the theoretically optimal sketching size $m^*$ does not have a closed-form representation, and $m^*$ can be picked by  minimizing the theoretical risk function across a set of values for $m$. Specifically, given fixed $p$ and $n$, we varied $\psi$ by taking a grid of $\psi \in (0, 1)$ with $|\psi_i - \psi_{i+1}| = \delta$ for $\delta=0.05$. This led to a set of potential values for $m^*$, i.e., $m_{i} = [\psi_{i} n]$. For each $m_i$, we calculated the negative solutions $c_0$ in \eqref{equation:a} and \eqref{equation:correlated_under} numerically using the functon \texttt{fsolve} in the Python library \texttt{SciPy}. These values of $c_0$ were then used to generate the theoretical risk curves in the overparameterized and underparameterized regime as described in Theorem \ref{thm:correlated_over} and Theorem \ref{correlated_under_biasvariance}, respectively. The optimal sketching size $m^*$ was selected as the one that minimized the theoretical risks across all $m_i$. With the optimal sketching size $m^*$, the empirical risks were calculated the same way as in Section \ref{appendix:numerical_fig2}. 

We further illustrate how the sketching size was selected using two examples shown in Figure \ref{fig:fig_5}. For $p = 200$ in the left panel, the risk attained minimum at $\psi \approx 1$ and $m^* = 400$, so  no sketching was needed. For $p = 424$ in the right panel, the risk attained minimum at $\psi \approx 0.6175$ and we set $m^* = 0.6175 \times 400 = 247$.

\section{Computational cost}\label{appendix:sec:computation}

We analyze the computational cost when the optimal sketching size $m^*$ is given 
 {\it a priori}. The time for the full sketching and orthogonal sketching (realized by the subsampled randomized Hadamard transform) is
\begin{align*}
    t_{\text{full}} = C_1np^2, ~ t_{\text{orthogonal}} = C_2pn\log n + C_3m^*p^2,
\end{align*}
where $C_1,C_2,C_3$ are some  constants. It is clear that the optimal orthogonal sketching can reduce computational costs when the condition $\frac{C_3}{C_1}\frac{m^*}{n} + \frac{C_2}{C_1}\frac{\log n}{p} < 1$ is satisfied. This condition is typically met in the overparameterized regime, where $p$ is large compared to $n$.

We conducted timing experiments on a MacMini with an Apple M1 processor and 16 GB of memory to measure the computational time required for the full-sample (no sketching) and sketched ridgeless least square estimators with orthogonal sketching (implemented through the subsampled randomized Hadamard transform) under  isotropic features. These experiments were designed to investigate the impact of the sketching size $m$, sample size $n$, and feature dimension $p$ on the computational time. To mitigate variations resulting from runtime disparities, we computed the average time from $10$ separate runs. 

Figure \ref{fig:fig_a1} compares the run time in seconds for different values of $p$ in both the underparameterized and overparameterized regimes, with a fixed sample size of $n = 10,000$. The figure demonstrates a significant computational benefit of sketching in the overparameterized regime. In this regime, as the feature dimension $p$ further deviates from the sample size $n$, sketching becomes increasingly time-efficient. Notably, when $p=11,500$, sketching saves time for almost every $\psi$. 

In Figure \ref{fig:fig_a2}, we fix the aspect ratio $\phi = p/n$  and the SNR to be $\SNR =\alpha/\sigma = 3 $ with $(\alpha, \sigma) = (6, 2)$, which implies a fixed optimal sketching ratio $\psi^*:= m^*/n$, while varying the sample size $n$. In this scenario, as the sample size $n$ increases, the optimal orthogonal sketching becomes even more time-efficient. This observation encourages the use of sketching when dealing with large sample sizes, which aligns with our intuition.

Figure \ref{fig:fig_a3} illustrates a scenario with a fixed aspect ratio $\phi$ and different $\SNR$, which results in smaller values for the optimal sketching size $m^*$. In such cases, employing the optimal orthogonal sketching significantly reduces computational time. 

In closing, we expect even larger computational improvements when using sketching in more complex models, such as neural networks, while simultaneously mitigating the out-of-time prediction risk. We leave  these experiments for future research.

\begin{figure}[t!]
    \centering
    \includegraphics[width=\columnwidth]{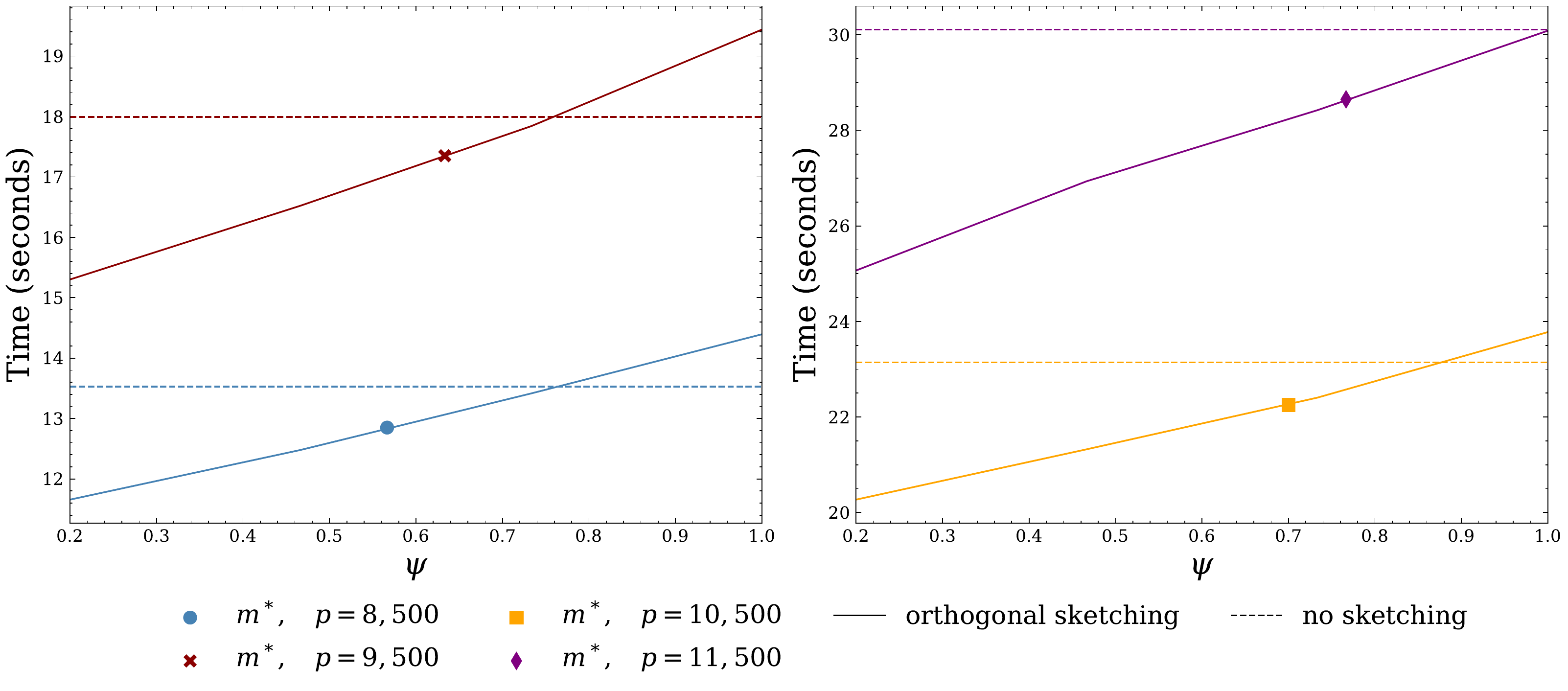}
    \caption{{Computational costs associated with the full-sample (no sketching) and sketched ridgeless least square estimators with orthogonal sketching under isotropic features, as functions of $\psi$, for a fixed sample size of $n=10,000$ and varying $p$.  The left panel shows the time required for the full-sample and orthogonally sketched estimators in the underparameterized regime, represented by the dotted and solid lines, respectively. The right panel depicts the time for estimators in the overparameterized regime. The dots and crosses mark the computational time required for the sketched estimators with the optimal sketching size $m^*$. We set $\SNR = \alpha/\sigma = 3$ with $(\alpha, \sigma) = (6, 2)$.} }
    \label{fig:fig_a1}
\end{figure}

\begin{figure}[t!]
    \centering
    \includegraphics[width=\columnwidth]{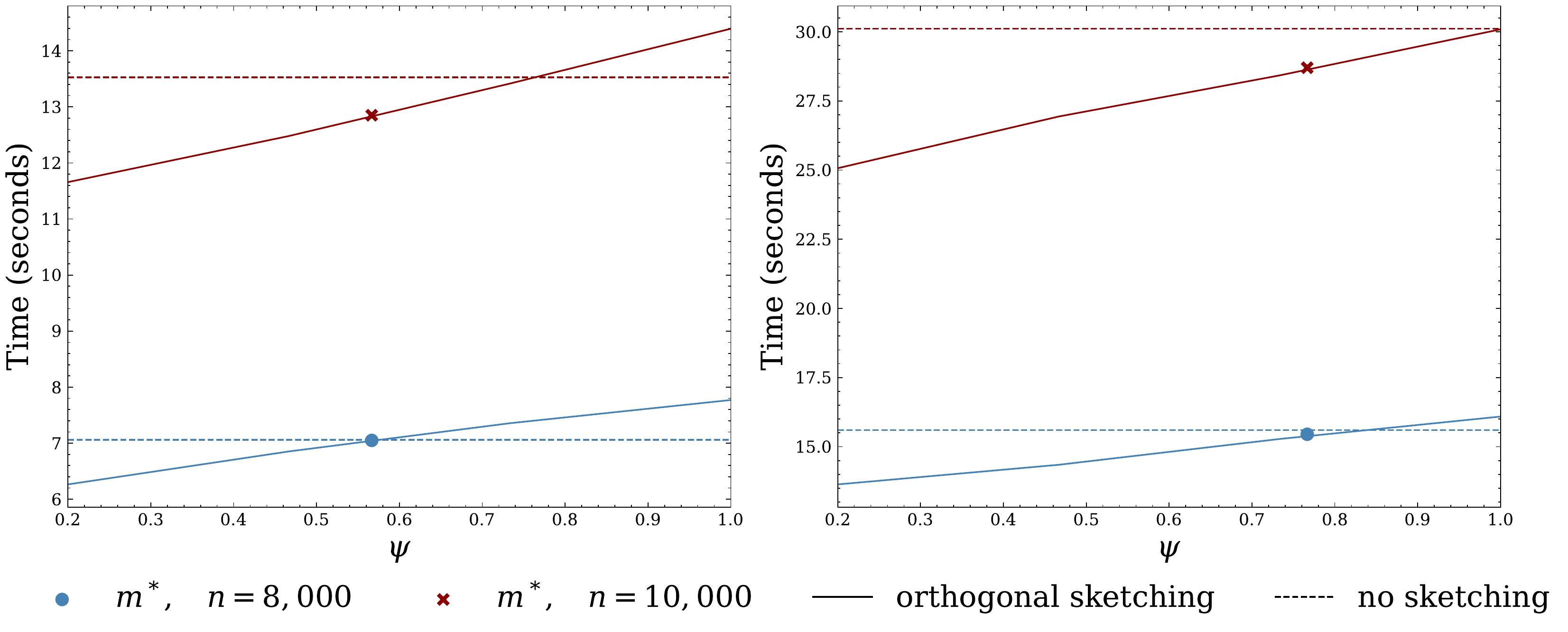}
    \caption{{Computational costs associated with the full-sample (no sketching) and sketched ridgeless least square estimators with orthogonal sketching under  isotropic features, as functions of $\psi$, for fixed $\phi=p/n$. The left panel shows the time required for the full-sample and orthogonally sketched estimators in the underparameterized regime with $\phi=0.85$, represented by the dotted and solid lines, respectively. The right panel depicts the time for estimators in the overparameterized regime  with $\phi=1.15$. The dots and crosses mark the computational time required for the sketched estimators with the optimal sketching size $m^*$. We set $\SNR = \alpha/\sigma = 3$ with $(\alpha, \sigma) = (6, 2)$.} }
    \label{fig:fig_a2}
\end{figure}

\begin{figure}[t!]
    \centering
    \includegraphics[width=\columnwidth]{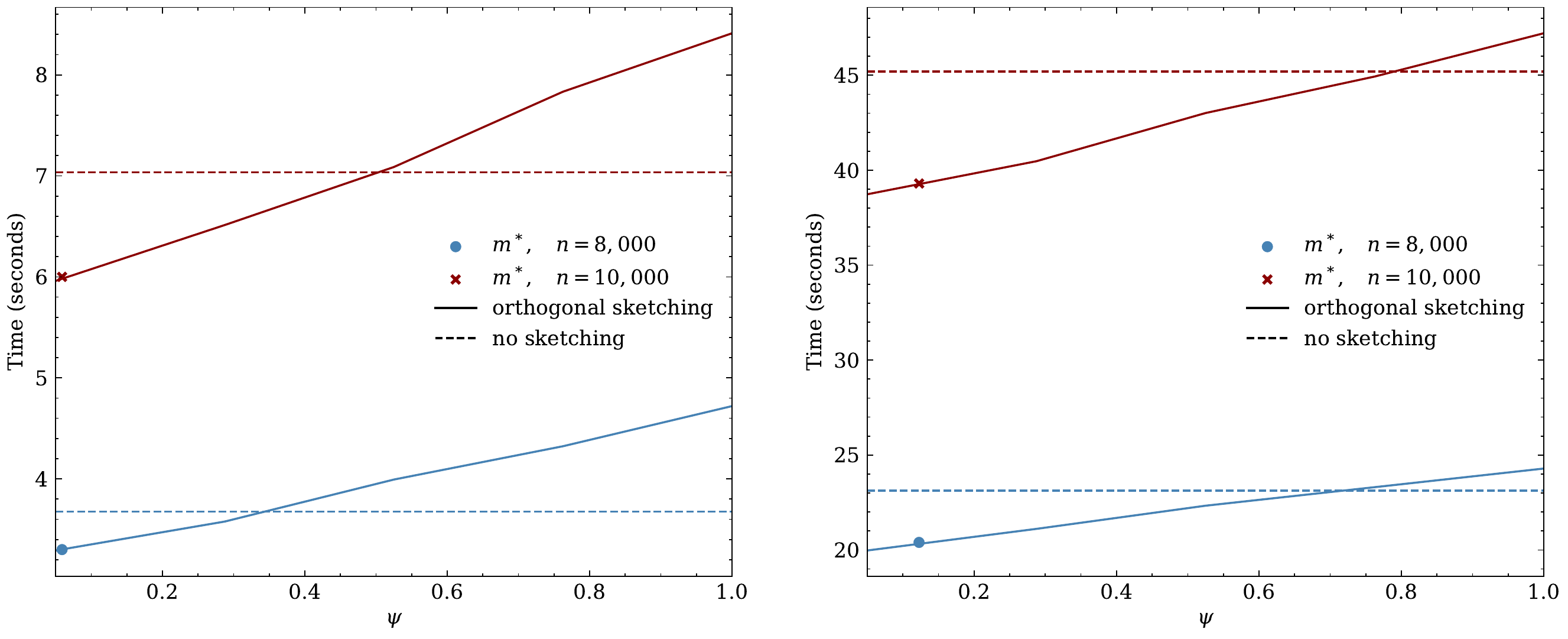}
    \caption{{Computational costs associated with the full-sample (no sketching) and sketched ridgeless least square estimators with  orthogonal sketching  under isotropic features , as functions of $\psi$, for fixed $\phi=p/n$. The left panel shows the time required for the full-sample and orthogonally sketched estimators in the underparameterized regime with $\phi=0.65$, represented by the dotted and solid lines, respectively. The right panel depicts the time for estimators in the overparameterized regime  with $\phi=1.35$. The dots and crosses mark the computational time required for the sketched estimators with optimal sketching size $m^*$. We set $\SNR = \alpha/\sigma = 1.1$, where $(\alpha, \sigma) = (22, 20)$.} }
    \label{fig:fig_a3}
\end{figure}

\section{Proofs for isotropic features}\label{appendix:sec:iso}

\subsection{Proof of Lemma \ref{lemma_bias_variance}}

Because  
\begin{align*}
\mathbb{E} \rbr{\hat\beta^S|\beta,S,X} &= (X^\top  S^\top  S X)^+ X^\top  S^\top  S X\beta, \\ 
{\rm Cov} \rbr{\hat\beta^S|\beta,S,X} = \sigma^2(&X^\top  S^\top  SX)^+ X^\top  S^\top  S S^\top  SX (X^\top  S^\top  SX)^+,    
\end{align*}
we can derive the expressions of $B_{(\beta,S,X)}\rbr{\hat{\beta}^S;\beta}$, $V_{(\beta, S,X)}\left(\hat\beta^{S}; \beta\right)$ and $V_{( S,X)}\left(\hat\beta^{S}; \beta\right)$ from their respective definitions. $B_{(S,X)}\rbr{\hat{\beta}^S;\beta}$ expression follows from the formula of the expectation of quadratic form and the fact that $(X^\top  S^\top  S X)^+ X^\top  S^\top  S X$ is idempotent.

Now since the eigenvalues of $I_p - (X^\top  S^\top  S X)^+ X^\top  S^\top  S  X $ are either 0 or 1, the eigenvalues of $\sbr{(X^\top  S^\top  S X)^+ X^\top  S^\top  S  X -I_p} \Sigma 
\sbr{(X^\top  S^\top  S X)^+ X^\top  S^\top  S  X -I_p}$ are uniformly bounded over $[0,C_1]$. Then \eqref{two_risk_equal_limit} can be obtained by applying   \cite[Lemma B.26]{bai2010spectral}. 

\subsection{Proof of Theorem \ref{thm:isotropic_limiting_risk}}

The proof for the underparameterized case directly follows from the Corollary \ref{Corollary:under_special_sketching} since the bias vanishes. As for the overparameterized case, according to  Theorem \ref{thm:correlated_over},  when $\Sigma = I_p$,  a simple calculation shows $c_0 = \psi\phi^{-1} -1$, and $B_{(S,X)}(\hat\beta^{S};\beta) \to \alpha^2 (1-\psi \phi^{-1})$, $V_{(S,X)}\left(\hat\beta^{S}; \beta\right) \to \sigma^2 (\phi\psi^{-1}-1)^{-1}$, which then leads to the desired results. We collect the proofs for Corollary \ref{Corollary:under_special_sketching} and Theorem \ref{thm:correlated_over} in Sections \ref{subsubsec:proof4_4} and \ref{subsubsec:proof4_2} respectively. 

\subsection{Proof of Theorem \ref{optimal_sketching_size}}\label{sketching_size}
We prove this theorem for orthogonal and \iid sketching separately. 

\noindent{\bf Orthogonal sketching}~ We start with orthogonal sketching first. Let
\$
f (x)=\alpha^2(1-x^{-1})+\frac{\sigma^2}{x-1},\ x>1.
\$ 
According to Theorem~\ref{thm:isotropic_limiting_risk}, for orthogonal sketching, both limiting risks in the overparameterized regime are $f(\phi\psi^{-1})$.

For the case of $\alpha \le \sigma$, i.e., ${\rm SNR}\le 1$, 
\$
f' (x)=\alpha^2 x^{-2} -\frac{\sigma^2}{(x-1)^2}
=\frac{(\alpha^2-\sigma^2)x^2-2\alpha^2 x +\alpha^2}{x^2(x-1)^2}
< 0,\ \forall x>1,
\$
and $\lim_{x\rightarrow \infty}f (x)=\alpha^2$. Thus, if $\phi \geq1$, $f(\phi \psi^{-1})$ decreases as $\psi$ decreases. If $\phi <1$, the limiting risks without sketching ($\psi=1$) are $\frac{\sigma^2\phi }{1-\phi}$, which exceeds $\alpha^2$ if and only if  $\phi > \frac{\alpha^2}{\alpha^2+\sigma^2}$. 
So the optimal sketching size is $m^*\ll n$ if $\phi > \frac{\alpha^2}{\alpha^2+\sigma^2}$, and $m^*=n$ otherwise.

For the case of $\alpha > \sigma$, i.e., ${\rm SNR}> 1$, $f(x)$ decreases when $x\in (1, \frac{\alpha}{\alpha-\sigma})$, and increases when $x\in [\frac{\alpha}{\alpha-\sigma}, \infty)$, and $f (\frac{\alpha}{\alpha -\sigma})=\sigma (2\alpha -\sigma)$. Thus, if $\frac{\alpha}{\alpha-\sigma} \geq \phi>1$,  orthogonal sketching can help reduce the limiting risks to  $\min_{x>1}f(x)$.  If $\phi<1$, the same improvement holds if and only if $\frac{\sigma^2 \phi }{1-\phi }>\sigma (2\alpha-\sigma)$, or equivalently $\phi>1-\frac{\sigma}{2\alpha}$. Thus, the optimal sketching size is $m^*=\phi \frac{\alpha-\sigma}{\alpha}\cdot n=\frac{\alpha-\sigma}{\alpha}\cdot p$ if $1-\frac{\sigma}{2\alpha}<\phi \le \frac{\alpha }{\alpha -\sigma}$; and $m^* = n$ otherwise.

{
\noindent{\bf \iid sketching}~
Because in the underparameterized case, sketching always increases the risk which follows the classic statistical intuition: a larger sample size is better. In other words, only sketching to the overparameterized case can help  reduce the risk. Because  \iid sketching shares the same limiting risks with orthogonal sketching in the overparameterized regime, it has the same optimal sketching size. 
}

\section{Proofs for correlated features}\label{appendix:proofs_risk}
\subsection{Proofs for the over-parameterized case}\label{appendix:proof_over}
\subsubsection{Proof of Lemma \ref{lemma:unique_neg_solution}}
Let $f(c) = 1 -  \int \frac{x}{-c+x\psi \phi^{-1}} \, dH(x)$. Because $f(0) = 1 - \phi\psi^{-1} <0$, $f(-\infty) = 1$, and $f$ is smooth, $f$ has at least one negative root. Suppose $c_1$ and $c_2$ are two negative roots with $c_1 < c_2$. Then we have 
\$
0 = f(c_1) -f(c_2) =  \int \frac{x(c_2-c_1)}{(-c_2+x\psi \phi^{-1})(-c_1+x\psi \phi^{-1})} \, dH(x) > 0,
\$
where the last inequality follows from the fact that the numerator and denominator are both larger than 0. This is a contradiction and thus $f$ has a unique negative root.

\subsubsection{Proof of Theorem \ref{thm:correlated_over}} \label{subsubsec:proof4_2}
\paragraph{Bias part}
To prove the bias part \eqref{equation:correlated_over_bias}, we first need some lemmas. Lemmas \ref{lemma:appendix_support} and \ref{lemma:appendix_sigmalambdamin} show that the minimal nonzero eigenvalues of conrresponding matrices are lower bounded, which will be used to guarantee to  exchange limits.  Lemma \ref{lemma:appendix_support} also proves, in the overparameterized case, $\frac{1}{p}SZZ^\top  S^\top $ is invertible almost surely for all large $n$. 

\begin{lemma}\label{lemma:appendix_support}
    Let $Z \in \RR^{n \times p}$ be a matrix with \text{i.i.d.} entries $Z_{ij}$ such that $\EE[Z_{ij}]=0$, $\EE[Z_{ij}^2] = 1$, and $\EE[Z_{ij}^4]<\infty$. Assume Assumption \ref{assumption_sketchingMatrix} and suppose $m,n,p \to \infty$ such that  $p/n \to \phi$, $m/n \to \psi \in (0,1)$. Then, there exists some constant $\tau >0$ such that almost surely for all large $n$, it holds that $\lambda_{\min}^+\rbr{\frac{1}{p}SZZ^\top  S^\top } = \lambda_{\min}^+\rbr{\frac{1}{p}Z^\top  S^\top  SZ} \geq \tau$, where $\lambda_{\min}^+$ denotes the smallest nonzero eigenvalue. 
    Furthermore, if (i) $\phi\psi^{-1} > 1$, then almost surely for all large $n$, $\frac{1}{p}SZZ^\top  S^\top $ is invertable; (2) $\phi\psi^{-1} < 1$, then almost surely for all large $n$, $\frac{1}{p}Z^\top  S^\top  SZ$ is invertible. 
\end{lemma}

\begin{proof}[Proof of Lemma \ref{lemma:appendix_support}]
Denote the limiting spectral measure of $\frac{1}{p}Z^\top  S^\top  SZ$ by $\mu$.   By \cite[Proposition 2.17]{yao2015sample}, the support of $\mu$ is completely determined by $\Psi(\alpha)$, known as the functional inverse of the function $a(x):=-1/s(x)$, where $s(x)$ is the Stieltjes transform of $\mu$. Specifically, if we let $\Gamma$ be the support of $\mu$, then $\Gamma^c \cap (0, \infty) = \{\Psi(a): \Psi'(a) > 0\}$. Under Assumption \ref{assumption_sketchingMatrix}, we can assume that the ESD of $S^\top  S$ converges to a nonrandom measure $\underline{B}$, which is the companion of $B$. Then,  $\Psi(a) = a + \phi^{-1}a\int \frac{t}{a-t} \, d\underline{B}(t)$, and hence $\Psi'(a) = 1+ \phi^{-1}\int \frac{t}{a-t} \, d\underline{B}(t) - \phi^{-1}a\int \frac{t}{(a-t)^2} \, d\underline{B}(t)$, which is smooth. 
    
    (i) If $\phi^{-1}\psi < 1$, then $\lim_{a\to 0^+}\Psi'(a) = 1 - \phi^{-1}(1-\underline{B}(\{0\}))  = 1 - \phi^{-1}\psi > 0$. Thus, there exists some small enough $\epsilon>0$ such that $\Psi$ is  increasing  on $(0, \epsilon)$. Besides, under Assumption \ref{assumption_sketchingMatrix}, the support of $\underline{B}$ is a subset of $\{0\}\cup [\tilde{C}_0, \tilde{C}_1]$. Thus, when $\epsilon$ is small enough, $\Psi$ is well defined on $(0, \epsilon)$. Since $\lim_{a\to 0 ^+}\Psi(a) = 0$ and $\Psi'$ is smooth, we know that there exists some $\tau > 0$ such that $\{\Psi(a): \Psi'(a) > 0\} \supseteq (0, \tau)$. 
    
    (ii) If $\phi^{-1}\psi > 1$, then $\lim_{a\to 0^-}\Psi'(a) <0$, and hence we can find some small enough $\epsilon$ such that $\Psi$ is decreasing on $(-\epsilon, 0)$. Since $\lim_{a \to -\infty}\Psi(a) = -\infty$, and by the smoothness of $\Psi$, we know $\{\Psi(a): \Psi'(a) > 0\} \supseteq (0, \Psi(-\epsilon))$. Overall, by combining these two situations and using  \cite[Theorem 1.1]{bai1998no}, we can show that there exists some $\tau > 0$ such that almost surely for all large $n$, $\lambda^+_{\min}(\frac{1}{p}Z^\top  S^\top  S Z) \geq \tau$.

To prove the invertibility of $\frac{1}{p}SZZ^\top  S^\top $ when $\phi\psi^{-1} > 1$, we first denote the limiting spectral measure of $\frac{1}{p}SZZ^\top  S^\top $ by $\underline{\mu}$. With  \cite[Proposition 2.2]{couillet2014analysis}, it holds that $\underline{\mu}\rbr{\{0\}} = 1 - \min\{1 - B\rbr{\{0\}}, \frac{n}{m}\min\{\frac{p}{n}, 1\}\} = 0$ since  $B\rbr{\{0\}} = 0$ under Assumption \ref{assumption_sketchingMatrix}. We thus obtain the invertibility of $\frac{1}{p}SZZ^\top  S^\top $ almost surely for all large $n$ by using again \cite[Theorem 1.1]{bai1998no}.   When $\phi\psi^{-1} < 1$, the invertibility of $\frac{1}{p}Z^\top  S^\top  SZ$ follows from a similar argument.
\end{proof}

\begin{lemma}\label{lemma:appendix_sigmalambdamin}
    Let $a,b >0$ be two positive constants. Let $A \in \RR^{n\times n}$ be a positive semidefinite matrix such that $\lambda_{\min}^+\rbr{A} \geq a$ and $\Sigma \in\RR^{n\times n} $  a positive definite matrix such that $\lambda_{\min} \rbr{\Sigma} \geq b$. Then, $\lambda_{\min}^+\rbr{\Sigma^{1/2}A\Sigma^{1/2}} \geq ab$. 
\end{lemma}

\begin{proof}[Proof of Lemma \ref{lemma:appendix_sigmalambdamin}]
The result follows from
\begin{align*}
\lambda_{\min}^+\rbr{\Sigma^{1/2}A\Sigma^{1/2}} &\geq \min_{x\in \RR^n:x^\top \Sigma^{1/2}A\Sigma^{1/2}x \neq 0} \frac{x^\top \Sigma^{1/2}A\Sigma^{1/2}x}{\nbr{x}^2}\\
& \geq \min_{x\in \RR^n:x^\top \Sigma^{1/2}A\Sigma^{1/2}x \neq 0} \frac{x^\top \Sigma^{1/2}A\Sigma^{1/2}x}{\nbr{\Sigma^{1/2}x}^2}\cdot \min_{x\in \RR^n:x \neq0} \frac{\nbr{\Sigma^{1/2}x}^2}{\nbr{x}^2}\\
&= ab.
\end{align*} 
\end{proof}

\begin{lemma}\label{lemma:appendix_equation}
    Assume Assumption \ref{assumption_general} and suppose $\phi,\psi>0$.  Then for any $z<0$,  the following equation \eqref{equation:appendix_b} has a unique negative solution $c(z) = c(z, \phi,\psi,H)$,
\begin{align}\label{equation:appendix_b}
    c(z) = \int \frac{(z+c(z))x}{-z-c(z)+x\psi \phi^{-1}} \, dH(x).
\end{align}
Furthermore, $\lim_{z\to 0^{-}}c(z) = c_0$ where $c_0$ is defined by \eqref{equation:a}.
\end{lemma}

\begin{proof}[Proof of Lemma \ref{lemma:appendix_equation}]
Given $z<0$, let $f(c(z)) = c(z) - \int \frac{(z+c(z))x}{-z-c(z)+x\psi\phi^{-1}} \, dH(x)$. We have $f(-\infty) = -\infty$ and  $f(0) = - \int \frac{zx}{-z+x\psi\phi^{-1}} \, dH(x)>0$ since $z<0$ and $x, \phi, \psi>0$. By the smoothness of $f$, we know $f$ has at least one negative solution. Suppose $c_1(z)$ and $c_2(z)$ are two negative solutions with $c_1(z)>c_2(z)$. Then, we have
\begin{align*}
    0 &=  \int \frac{x(\frac{z}{c_1(z)}+1)}{-z-c_1(z)+x\psi\phi^{-1}}\, dH(x) - \int \frac{x(\frac{z}{c_2(z)}+1)}{-z-c_2(z)+x\psi\phi^{-1}}\, dH(x)\\
    &=  \int \Bigg\{ \frac{x(c_1(z)-c_2(z))}{(-z-c_1(z)+x\psi\phi^{-1})(-z-c_2(z)+x\psi\phi^{-1})}  + \frac{z^2x(\frac{c_1(z)-c_2(z)}{c_1(z)c_2(z)})}{(-z-c_1(z)+x\psi\phi^{-1})(-z-c_2(z)+x\psi\phi^{-1})} \\
    &+ \frac{zx^2\psi\phi^{-1}(\frac{c_2(z)-c_1(z)}{c_1(z)c_2(z)})}{(-z-c_1(z)+x\psi\phi^{-1})(-z-c_2(z)+x\psi\phi^{-1})} + \frac{zx(\frac{(c_1(z)-c_2(z))(c_1(z)+c_2(z))}{c_1(z)c_2(z)})}{(-z-c_1(z)+x\psi\phi^{-1})(-z-c_2(z)+x\psi\phi^{-1})} \Bigg\} \,dH(x).
\end{align*}
Since $z, c_1(z), c_2(z) <0$, it is easy to find that each term above is larger than 0. This contradiction shows for given $z<0$, \eqref{equation:appendix_b} has a unique negative solution, denoted by $c(z)$.

Next, we show $\lim_{z\to 0^{-}}c(z) = c_0$. Given $z<0$, let 
\begin{align*}
    g(a, z) = 1 -  \int \frac{x}{-z-a+x\psi\phi^{-1}} \, dH(x) -  \int \frac{{zx}/{a}}{-z-a+x\psi\phi^{-1}}\, dH(x).
\end{align*}
Since $c(z)$ is the solution of \eqref{equation:appendix_b}, $g(c(z), z) = 0$. For any small $\epsilon > 0$, we can find a sufficiently small $\delta_1 >0$ such that, when $-\delta_1 < z < 0$, we have $0 < -z-c_0-\epsilon + x\psi\phi^{-1} < -c_0 + x\psi\phi^{-1}$. The second inequality is satisfied by taking $\delta_1 < \epsilon$. Because $x$ lies in the support of the measure $H$,  $x>C_0>0$. Moreover, since $c_0<0$, the first inequality holds when $\epsilon$ and $\delta_1$ are sufficiently small. Because $c_0$ is the root of $g$ when $z=0$,  in this case, we have  $g(c_0 + \epsilon, z) < 0$ when $z \in (-\delta_1,0)$. Furthermore,  we can find a sufficiently small $\delta_2>0$ such that $g(c_0 - \epsilon, z) > 0$ when $z \in (-\delta_2,0)$. This is because  $g(c_0,0)=0$, $1-  \int \frac{x}{-c_0+\epsilon+x\psi\phi^{-1}}\, dH(x) > 0$, and $\lim_{z\to 0^{-}} \int \frac{{zx}/{a}}{-z-a+x\psi\phi^{-1}}\, dH(x) = 0$.

To conclude, taking $\delta = \min \{\delta_1,\delta_2\}$, we have, when $z \in (-\delta,0)$, $g(c_0 + \epsilon, z) < 0$ and $g(c_0 - \epsilon, z) > 0$. By the smoothness of $g(a,z)$ with respect to $a$, we know $c(z) \in (c_0-\epsilon, c_0+\epsilon)$. Using the definition of a limit completes the proof. 

\end{proof}

Now we prove the bias part \eqref{equation:correlated_over_bias}.
\begin{proof}[Proof of the bias part  \eqref{equation:correlated_over_bias}]
For all large $n$, we have almost surely
\begin{align}\label{equation:appendix_biaslemma_1}
    &\rbr{X^\top  S^\top  S X}^+ X^\top  S^\top  S  X \notag\\  
    =~& \rbr{S X}^+S X \notag\\
    =~& \lim_{\delta\to 0^+}X^\top  S^\top  \rbr{SX X^\top  S^\top  + \delta I_m}^{-1} S X \notag\\
    =~& X^\top  S^\top  \rbr{SX X^\top  S^\top  }^{-1} S X,
\end{align}
where the first equality uses   $A^+ = \rbr{A^\top  A}^+ A^\top $ for any matrix $A$, the second inequality uses  $A^+ = \lim_{\delta \to 0^+}A^\top  \rbr{AA^\top  + \delta I}^{-1}$, and the third equality follows from Lemma \ref{lemma:appendix_support} and Assumption \ref{assumption_general}.  Specifically, when $SZ Z^\top  S^\top $ is invertible and $\Sigma$ satisfies  Assumption \ref{assumption_general}, then $SX X^\top  S^\top  = SZ \Sigma Z^\top  S^\top $ is also invertible. Let the  singular value
decomposition (SVD) of $S$ be $S = U D V$ where $U \in \RR^{m \times m}$, $V \in \RR^{m \times n}$ are both orthogonal matrices,  $D \in \RR^{m \times m}$ is a diagonal matrix. By Assumption \ref{assumption_sketchingMatrix}, we know almost surely for all large $n$, $D$ is invertible.  Then the RHS (right hand side) of  \eqref{equation:appendix_biaslemma_1} can be writen as 
\begin{align}\label{equation:appendix_biaslemma_2}
    X^\top  S^\top  \rbr{SX X^\top  S^\top  }^{-1} S X = X^\top  V^\top  \rbr{V X X^\top  V^\top  }^{-1} V X = \rbr{X^\top  V^\top  V X}^+ X^\top  V^\top  V  X.
\end{align}
Thus, by \eqref{equation:appendix_biaslemma_1}, \eqref{equation:appendix_biaslemma_2} and Lemma \ref{lemma_bias_variance}, we have
\begin{align}\label{equation:appendix_biaslemma_3}
B_{(S,X)}(\hat\beta^{S};\beta)
&= \frac{\alpha^2}{p}\trace\cbr{\sbr{I_p -\rbr{\frac{1}{p} X^\top  V^\top  V X}^+ \frac{1}{p} X^\top  V^\top  V  X }\Sigma}. 
\end{align}
For any $z<0$,
\begin{align*}
&\abr{\frac{1}{p}\trace\sbr{\rbr{\frac{1}{p} X^\top  V^\top  V X}^+ \frac{1}{p} X^\top  V^\top  V  X \Sigma}  - \frac{1}{p}\trace\sbr{\rbr{\frac{1}{p}X^\top  V^\top  V X - zI_p}^{-1} \frac{1}{p}X^\top  V^\top  V  X \Sigma}}\\
\leq~&  \frac{\abr{z} \nbr{\Sigma}_2}{\lambda_{\min}^+ \rbr{\frac{1}{p}X^\top  V^\top  V X}-z}\\
=~ & \frac{\abr{z} \nbr{\Sigma}_2}{\lambda_{\min}^+ \rbr{\frac{1}{p} \Sigma^{1/2}Z^\top  V^\top  V Z\Sigma^{1/2}}-z} \leq  \frac{\abr{z}C_1}{C_0\tau -z},
\end{align*}
where the last inequality follows from Lemmas \ref{lemma:appendix_support}, \ref{lemma:appendix_sigmalambdamin} and Assumption \ref{assumption_general}. Thus, taking limites on both sides of \eqref{equation:appendix_biaslemma_3} gives
\begin{align}\label{equation:appendix_biaslemma_4}
\lim_{n\to \infty}B_{(S,X)}(\hat\beta^{S};\beta) 
&= \alpha^2 \lim_{n\to \infty}\lim_{z \to 0^-} \frac{1}{p}\trace\cbr{\sbr{I_p - \rbr{\frac{1}{p}X^\top  V^\top  V X - z I_p}^{-1} \frac{1}{p}X^\top  V^\top  V  X }\Sigma} \notag\\
& = \alpha^2 \lim_{n\to \infty}\lim_{z \to 0^-} \frac{1}{p}\trace\cbr{\sbr{I_p - \rbr{\frac{1}{p}X^\top  V^\top  V X - z I_p}^{-1} \left(\frac{1}{p}X^\top  V^\top  V  X -z I_p + zI_p \right) }\Sigma} \notag\\
& = -\alpha^2 \lim_{n\to \infty}\lim_{z \to 0^-}z \frac{1}{p}\trace\sbr{\rbr{\frac{1}{p}X^\top  V^\top  V X - z I_p}^{-1}\Sigma }. 
\end{align}
Now we can follow a similar argument to the proof of Theorem 1 in \cite{hastie2019surprises} to show the validity of exchanging the limits  $n \to \infty$ and $z \to 0^-$. Define $f_n(z) = -\frac{z}{p}\trace\sbr{\rbr{\frac{1}{p}X^\top  V^\top  V X - z I_p}^{-1} \Sigma}$. Since $|f_n(z)| \leq |z|\nbr{(\frac{1}{p}X^\top  V^\top  V X - z I_p)^{-1}}_2\nbr{\Sigma}_2 \leq C_1$, we know $f_n(z)$ is uniformly bounded. Besides,
\begin{align}
|f_n'(z)| &\leq \frac{1}{p}\abr{\trace\sbr{\rbr{\frac{1}{p}X^\top  V^\top  V X - z I_p}^{-1}\Sigma } + z~\trace\sbr{\rbr{\frac{1}{p}X^\top  V^\top  V X - z I_p}^{-2} \Sigma}} \notag\\
&\leq  \frac{\lambda^+_{\min}\rbr{\frac{1}{p}X^\top  V^\top  V X} \nbr{\Sigma}_2}{\sbr{\lambda^+_{\min}\rbr{\frac{1}{p}X^\top  V^\top  V X}-z}^2} \notag\\
&\leq \frac{\nbr{\Sigma}_2 }{\lambda^+_{\min}\rbr{\frac{1}{p}X^\top  V^\top  V X}} \leq \frac{ C_1}{C_0 \tau},  \label{ineq_fprime}
\end{align}
where the last inequality holds almost surely for all large $n$. As its derivatives are bounded, the sequence $\{f_n\}_{n=1}^\infty$ is equicontinuous, and hence, by the Arzela-Ascoli theorem, $f_n$ converges uniformly. Thus, we can use Moore-Osgood theorem to conclude the validity of exchanging the limits $n \to \infty$ and $\lambda \to 0^-$. Define
\begin{align}\label{definition:appendix_m}
    m_{1n}(z) = \frac{1}{p}\trace \sbr{\rbr{\frac{1}{p}\Sigma^{1/2}Z^\top  V^\top  V Z \Sigma^{1/2} - z I_p}^{-1}\Sigma}, \quad m_{2n}(z) = \frac{1}{p}\trace\sbr{\rbr{\frac{1}{p} V Z \Sigma Z^\top  V^\top  - z I_m}^{-1} }.
\end{align}
According to \cite{zhang2007spectral}, almost surely, as $n\to \infty$, $m_{1n}(z) \to m_1(z)$, and $m_{2n}(z) \to m_2(z)$ where for given $z<0$,  $\rbr{m_1(z), m_2(z)} \in \RR^+ \times \RR^+$ is the unique solution of the self-consistent equations
\begin{align}\label{equations:appendix_selfconsistent}
    m_1(z) &= \int \frac{x}{-z\sbr{1+xm_2(z)}} \, dH(x), \notag\\
    m_2(z) &= \psi \phi^{-1} \frac{1}{-z\sbr{1+m_1(z)}}. 
\end{align}
Substituting $m_2$ into $m_1$ in \eqref{equations:appendix_selfconsistent} and mutiplying both sides by $z$, we obtain
\begin{align*}
    zm_1(z) = \int \frac{\rbr{z + zm_1(z)}x}{-z - z m_1(z) +x\psi\phi^{-1}} \, dH(x).
\end{align*}
By lemma \ref{lemma:appendix_equation}, we know $\lim_{z\to 0^-} zm_1(z) = c_0$. Exchanging the limits in \eqref{equation:appendix_biaslemma_4},  we have almost surely
\begin{align*}
    \lim_{n\to \infty}B_{(S,X)}(\hat\beta^{S};\beta) &= -\alpha^2 \lim_{z \to 0^-}\lim_{n\to \infty} z \frac{1}{p}\trace\sbr{\rbr{\frac{1}{p}X^\top  V^\top  V X - z I_p}^{-1}\Sigma }\\
    &= -\alpha^2\lim_{z\to 0^-} zm_1(z) = -\alpha^2 c_0.
\end{align*}
Lemma \ref{lemma_bias_variance} assures that  $B_{(\beta, S,X)}(\hat\beta^{S};\beta)$ converges almost surely to the same limit.
\end{proof}

\paragraph{Variance part}
To prove the variance part \eqref{equation:correlated_over_variance}, we need the following theorem, often known as the Vitali’s theorem \citep[Lemma 2.14]{bai2010spectral}. This theorem ensures the convergence of the derivatives of converging analytic functions.
\begin{lemma}[Vitali’s convergence theorem]\label{lemma:appendix_vitali}
    Let $f_1, f_2, \cdots$ be analytic on the domain $D$, satisfying $\abr{f_n(z)} \leq M$ for every $n$ and $z \in D$. Suppose that there is an analytic function $f$ on $D$ such that $f_n(z) \to f(z)$ for all $z \in D$. Then it also holds that $f_n'(z) \to f'(z)$ for all $z\in D$. 
\end{lemma}

\begin{proof}[Proof of the variance part \eqref{equation:correlated_over_variance}]

Let the  singular value
decomposition (SVD) of $S$ be $S = U D V$ where $U \in \RR^{m \times m}$, $V \in \RR^{m \times n}$ are both orthogonal matrices, $D \in \RR^{m \times m}$ is a diagonal matrix. According to Lemma \ref{lemma_bias_variance},
\begin{align}\label{equation:appendix_correlatedvar_1}
V_{(S,X)}\left(\hat\beta^{S}; \beta\right)
& = \sigma^2{\rm tr}\sbr{ \rbr{X^\top  S^\top  SX}^+ X^\top  S^\top  S S^\top  SX \rbr{X^\top  S^\top  SX}^+ \Sigma}\notag\\
& = \sigma^2{\rm tr}\sbr{ \rbr{SX}^+ S S^\top   \rbr{X^\top  S^\top }^+ \Sigma} \notag\\ 
& = \sigma^2{\rm tr}\sbr{\lim_{\delta\rightarrow 0+}X^\top  S^\top  \left(S X X^\top  S^\top +\delta I_m \right)^{-1} SS^\top  \left(S X X^\top  S^\top +\delta I_m \right)^{-1} SX \Sigma }\notag\\
& = \sigma^2{\rm tr}\sbr{X^\top  S^\top  \left(S X X^\top  S^\top  \right)^{-1} SS^\top  \left(S X X^\top  S^\top  \right)^{-1} SX \Sigma }\notag\\
& = \sigma^2{\rm tr}\sbr{X^\top  V^\top  \left(V X X^\top  V^\top  \right)^{-2}   VX \Sigma }\notag\\
& = \sigma^2{\rm tr}\sbr{ \rbr{VX}^+ \rbr{X^\top  V^\top }^+ \Sigma }\notag\\
& = \sigma^2{\rm tr}\sbr{  \rbr{X^\top  V^\top  V X}^+ \Sigma }\notag\\
& = \frac{\sigma^2}{p}{\rm tr}\sbr{  \rbr{\frac{1}{p}X^\top  V^\top  V X}^+ \Sigma },
\end{align}
where similar to the proof of the bias part in Theorem \ref{thm:isotropic_limiting_risk}, we use the identity $A^+ = \rbr{A^\top  A }^+A^\top  = \lim_{\delta\to 0^+}A^\top  \rbr{AA^\top  + \delta I}^{-1}$ for any matrix $A$, and the fact that almost surely for all large $n$, $S X X^\top  S^\top $ is invertible. Define 
\begin{align*}
    g_n(z) = \frac{1}{p}  {\rm tr}\sbr{ \rbr{\frac{1}{p}X^\top  V^\top  V X} \rbr{\frac{1}{p}X^\top  V^\top  V X - zI_p}^{-2} \Sigma }.
\end{align*}
Since for any $z \leq 0, x > 0$, we have 
\$
\abr{\frac{x}{(x-z)^2} - \frac{1}{x}} \leq \frac{2\abr{z}}{x^2}. 
\$
Thus, by Lemma \ref{lemma:appendix_sigmalambdamin} and Assumption \ref{assumption_general}, for any $z<0$,
\begin{align}\label{inequality:appendix_varapprox}
    \abr{\frac{V_{(S,X)}\left(\hat\beta^{S}; \beta\right)}{\sigma^2} - g_n(z)} \leq  \frac{2\abr{z}}{\sbr{\lambda_{\min}^+\rbr{\frac{1}{p}X^\top  V^\top  V X}}^2} \nbr{\Sigma}_2 \leq \frac{2\abr{z}C_1}{\rbr{C_0\tau}^2}.
\end{align}
By \eqref{inequality:appendix_varapprox}, we can continue \eqref{equation:appendix_correlatedvar_1},
\begin{align}\label{equation:appendix_correlatedvar_2}
    &\lim_{n\to \infty} V_{(S,X)}\left(\hat\beta^{S}; \beta\right)\notag\\
    =~& \sigma^2 \lim_{n\to \infty} \lim_{z \to 0^-}\frac{1}{p} {\rm tr}\sbr{ \rbr{\frac{1}{p}X^\top  V^\top  V X} \rbr{\frac{1}{p}X^\top  V^\top  V X - zI_p}^{-2} \Sigma }\notag\\
    =~& \sigma^2 \lim_{n\to \infty} \lim_{z \to 0^-}\frac{1}{p} {\rm tr}\sbr{ \rbr{\frac{1}{p}X^\top  V^\top  V X - zI_p + zI_p} \rbr{\frac{1}{p}X^\top  V^\top  V X - zI_p}^{-2} \Sigma }\notag\\
    =~& \sigma^2 \lim_{n\to \infty} \lim_{z \to 0^-} \frac{1}{p} {\rm tr}\sbr{  \rbr{\frac{1}{p}X^\top  V^\top  V X - zI_p}^{-1} \Sigma } + \frac{1}{p} {\rm tr}\sbr{ z \rbr{\frac{1}{p}X^\top  V^\top  V X - zI_p}^{-2} \Sigma }.
\end{align}
We now verify the validity of exchanging the limits $n \to \infty$ and $z \to 0^-$. As in the proof of the bias part of Theorem \ref{thm:isotropic_limiting_risk}, in order to use Arzela-Ascoli theorem and Moore-Osgood theorem, it suffices to show $g_n(z)$ and $g_n'(z)$ are both uniformly bounded. We know it holds almost surely for all large $n$ that for any $z<0$,
\begin{align*}
    \abr{g_n(z)} \leq \frac{\nbr{\frac{1}{p}X^\top  V^\top  V X}_2\nbr{\Sigma}_2}{\sbr{\lambda_{\min}^+\rbr{\frac{1}{p}X^\top  V^\top  V X}-z}^2}\leq \frac{\nbr{\frac{1}{p}Z Z^\top  }_2 \nbr{V^\top  V}_2 \nbr{\Sigma}_2^2 }{\sbr{\lambda^+_{\min}\rbr{\frac{1}{p}X^\top  V^\top  V X}-z}^2} \leq \frac{\rbr{1+\sqrt{\phi^{-1}}}^2 C_1^2}{(C_0 \tau)^2}.
\end{align*}
Moreover, 
\begin{align*}
    \abr{g_n'(z)}  &= \abr{\frac{2}{p} {\rm tr}\sbr{  \rbr{\frac{1}{p}X^\top  V^\top  V X - zI_p}^{-2} \Sigma } + \frac{2}{p} {\rm tr}\sbr{ z \rbr{\frac{1}{p}X^\top  V^\top  V X - zI_p}^{-3} \Sigma }}\\
    &\leq \frac{2\nbr{\frac{1}{p}X^\top  V^\top  V X}_2\nbr{\Sigma}_2}{\sbr{\lambda_{\min}^+\rbr{\frac{1}{p}X^\top  V^\top  V X}-z}^3} \leq \frac{2\rbr{1+\sqrt{\phi^{-1}}}^2 C_1^2}{(C_0 \tau)^3}.
\end{align*}
Thus, $g_n(z)$ and $g_n'(z)$ are both uniformly bounded and hence, we can exchange the limits. Recall the definition of $m_{1n}(z)$ in \eqref{definition:appendix_m}, we know $g_n(z) = m_{1n}(z) + zm_{1n}'(z) = (zm_{1n}(z))'$. We will use Lemma \ref{lemma:appendix_vitali} to show $g_n(z) \to m_1(z) + zm_{1}'(z)$ almost surely as $n \to \infty$. Since $zm_{1n}(z)$ and $zm_{1}(z)$ are analytic on $(-\infty, 0)$ such that  $zm_{1n}(z)\to zm_{1}(z)$; see \cite{zhang2007spectral}.  In addition, as in the proof of the bias part of Theorem \ref{thm:isotropic_limiting_risk}, almost surely for all large 
$n$, it holds that $\abr{zm_{1n}(z)} \leq C_1$. Thus the conditions of Lemma \ref{lemma:appendix_vitali} are satisfied. By exchanging the limits in \eqref{equation:appendix_correlatedvar_2} and using Lemma \ref{lemma:appendix_vitali}, we obtain
\begin{align}\label{equation:appendix_correlatedvar_3}
     \lim_{n\to \infty} V_{(S,X)}\left(\hat\beta^{S}; \beta\right) &= \sigma^2\lim_{z \to 0^-} \lim_{n\to \infty} g_n(z)
     = \sigma^2 \lim_{z\to 0^-}m_1(z) + zm_1'(z).
\end{align}
Recall the self-consistent equations in \eqref{equations:appendix_selfconsistent}. A direct calculation yields
\begin{align*}
    m_1(z) + zm_1'(z) = \int \frac{\rbr{1+m_1(z)+zm_1'(z)}x^2\psi\phi^{-1}}{(z + zm_1(z)-x\psi\phi^{-1})^2} \, dH(x).
\end{align*}
Taking $z\to 0^-$ in the above equality and using $\lim_{z\to 0^-}zm_1(z) = c_0$ in Lemma \ref{lemma:appendix_equation}, we  derive 
\begin{align*}
    \lim_{z\to 0^-}m_1(z) + zm_1'(z) = \frac{\int \frac{x^2\psi\phi^{-1}}{\rbr{c_0 - x\psi\phi^{-1}}^2}\, dH(x)}{1 -\int \frac{x^2\psi\phi^{-1}}{\rbr{c_0 - x\psi\phi^{-1}}^2}\, dH(x)}.
\end{align*}
Combining the above limit with \eqref{equation:appendix_correlatedvar_3}, we have almost surely
\begin{align*}
     \lim_{n\to \infty} V_{(S,X)}\left(\hat\beta^{S}; \beta\right) = \sigma^2\frac{\int \frac{x^2\psi\phi^{-1}}{\rbr{c_0 - x\psi\phi^{-1}}^2}\, dH(x)}{1 -\int \frac{x^2\psi\phi^{-1}}{\rbr{c_0 - x\psi\phi^{-1}}^2}\, dH(x)}.
\end{align*}
 Lemma \ref{lemma_bias_variance} assures that $V_{(\beta, S,X)}\left(\hat\beta^{S}; \beta\right)$ converges almost surely to the same limit.
\end{proof}

\subsection{Proofs for the under-parameterized case}\label{appendix:proof_under}

\subsubsection{Proof of Theorem \ref{correlated_under_biasvariance}}
According to  Lemma \ref{lemma:appendix_support} and Assumption \ref{assumption_general}, we know almost surely for all large $n$, $\frac{1}{p}X^\top  S^\top  SX$ is invertible. Thus by Lemma \ref{lemma_bias_variance}, it holds that almost surely for all large $n$, $B_{(S,X)}(\hat\beta^{S};\beta) = B_{(\beta, S,X)}(\hat\beta^{S};\beta) = 0$ and hence \eqref{limit:correlated_under_bias} holds. To show the limiting variance \eqref{equation:correlated_under_variance}, we follow from a similar proof to that for  the bias part in Theorem \ref{thm:isotropic_limiting_risk}. To be concise, we only sketch the proof here.  Similar to \eqref{equation:appendix_correlatedvar_1} and  \eqref{equation:appendix_correlatedvar_2}, we have
\begin{align}\label{equation:appendix_under_correlatedvar_1}
&V_{(S,X)}\left(\hat\beta^{S}; \beta\right)\notag\\
 =~& \sigma^2{\rm tr}\sbr{ \rbr{X^\top  S^\top  SX}^{-1} X^\top  S^\top  S S^\top  SX \rbr{X^\top  S^\top  SX}^{-1} \Sigma}\notag\\
 =~& \sigma^2{\rm tr}\sbr{ \rbr{Z^\top  S^\top  SZ}^{-1} Z^\top  S^\top  S S^\top  SZ \rbr{Z^\top  S^\top  SZ}^{-1} }\notag\\
 =~& \sigma^2{\rm tr}\sbr{ \rbr{Z^\top  S^\top }^+ \rbr{SZ}^+S S^\top    } \notag\\ 
 =~& \sigma^2{\rm tr}\sbr{  \rbr{SZZ^\top  S^\top  }^+ SS^\top  }\notag\\
 =~& \frac{\sigma^2}{n}{\rm tr}\sbr{  \rbr{\frac{1}{n}SZZ^\top  S^\top  }^+ SS^\top  }\notag\\
=~& \sigma^2  \lim_{z \to 0^-}\frac{1}{n} {\rm tr}\sbr{ \rbr{\frac{1}{n} SZ Z^\top  S^\top } \rbr{\frac{1}{n}SZ Z^\top  S^\top  - zI_m}^{-2} SS^\top  }\notag\\
=~& \sigma^2  \lim_{z \to 0^-} \frac{1}{n} {\rm tr}\sbr{  \rbr{\frac{1}{n}SZ Z^\top  S^\top  - zI_m}^{-1} SS^\top  } + \frac{1}{n} {\rm tr}\sbr{ z \rbr{\frac{1}{n}SZ Z^\top  S^\top  - zI_m}^{-2} SS^\top }.
\end{align}
Define
\begin{align}\label{definition:appendix_m_tilde}
    \tilde{m}_{1n}(z) = \frac{1}{n}\trace \sbr{\rbr{\frac{1}{n}SZ Z^\top  S^\top  - zI_m}^{-1}SS^\top }, \, \tilde{m}_{2n}(z) = \frac{1}{n}\trace\sbr{\rbr{\frac{1}{n}  Z^\top  S^\top  SZ - zI_p}^{-1} }.
\end{align}
Then $\tilde{m}_{1n}(z) \to \tilde{m}_1(z)$ and $\tilde{m}_{2n}(z) \to \tilde{m}_2(z)$ almost surely as $n\rightarrow \infty$, where  $\rbr{\tilde{m}_1(z), \tilde{m}_2(z)} \in \RR^+ \times \RR^+$ is the unique solution of the self-consistent equations \citep{zhang2007spectral}  
\begin{align}\label{equations:appendix_under_selfconsistent}
    \tilde{m}_1(z) &= \psi \int \frac{x}{-z\sbr{1+x\tilde{m}_2(z)}} \, dB(x), \notag\\
    \tilde{m}_2(z) &= \phi  \frac{1}{-z\sbr{1+\tilde{m}_1(z)}},
\end{align}
for any $z<0$. 
Substituting $\tilde{m}_2$ into $\tilde{m}_1$ in \eqref{equations:appendix_under_selfconsistent} and multiplying both sides by $z$, we obtain
\begin{align*}
    z\tilde{m}_1(z) = \int \psi \frac{\rbr{z + z\tilde{m}_1(z)}x}{-z - z \tilde{m}_1(z) +x\phi} \, dH(x).
\end{align*}
Following the similar proofs to Lemma \ref{lemma:appendix_equation} and the bias part in Theorem \ref{thm:isotropic_limiting_risk}, we can obtain $\lim_{z\to 0^-}z\tilde{m}_1(z) = \tilde{c}_0$ where $\tilde{c}_0$ is defined in \eqref{equation:correlated_under}. Following the same argument for verifying interchange of the limits and Lemma \ref{lemma:appendix_vitali}, we have 
\begin{align*}
        &\lim_{n\to \infty}V_{(S,X)}(\hat\beta^{S};\beta) = \lim_{n \to \infty}V_{(\beta, S,X)}(\hat\beta^{S};\beta)\\
        =~& \sigma^2\lim_{z\to 0^-}\tilde{m}_1(z) + z\tilde{m}'_1(z) = \sigma^2\frac{\psi\int \frac{x^2\phi}{\rbr{\tilde{c}_0-x\phi}^2} \, dB(x)}{1-\psi\int \frac{x^2\phi}{\rbr{\tilde{c}_0-x\phi}^2} \, dB(x)}.
\end{align*}

\subsubsection{Proof of Corollary \ref{Corollary:under_special_sketching}}\label{subsubsec:proof4_4}
When $S$ is an orthogonal sketching matrix, i.e., $SS^\top  = I_m$, we have $B(x) = \delta_{\{1\}}(x)$ where $\delta$ is the Dirac function. A simple calculation shows $\tilde{c}_0 = \phi - \psi$, and hence
\begin{align*}
        V_{(S,X)}(\hat\beta^{S};\beta) = V_{(\beta, S,X)}(\hat\beta^{S};\beta)
         \to \sigma^2\frac{\psi\int \frac{x^2\phi}{\rbr{\tilde{c}_0-x\phi}^2} \, dB(x)}{1-\psi\int \frac{x^2\phi}{\rbr{\tilde{c}_0-x\phi}^2} \, dB(x)} = \sigma^2\frac{\phi\psi^{-1}}{1-\phi\psi^{-1}}.
\end{align*}
When $S$ is an \text{i.i.d.} sketching matrix, we know that almost surely, the ESD of $SS^\top $ converges to the M-P law with parameter $\psi$, whose CDF (cumulative distribution function) is denoted by $F_{\psi}$, i.e., $B = F_{\psi}$. The self-consistent equation \eqref{equation:correlated_under} reduces to 
\$
1 = \frac{\psi}{\phi} + \frac{\psi \tilde{c}_0}{\phi^2}s_{\psi}\rbr{\frac{\tilde{c}_0}{\phi}},
\$ 
where $s_{\psi}$ is the Stieltjes transform of M-P law with parameter $\psi$. According to the seminal work \citep{marvcenko1967distribution}, we know for any $z<0$,
\begin{align*}
    s_\psi(z) = \frac{1-\psi-z-\sqrt{(z-1-\psi)^2-4\psi}}{c\psi z}.
\end{align*}
A direct calculation shows $\tilde{c}_0 = -\psi -\phi^2 + \phi + \psi \phi$. Furthermore, 
\begin{align*}
    \psi\int \frac{x^2\phi}{\rbr{\tilde{c}_0-x\phi}^2} \, dB(x) &= \psi \phi^{-1} \int \frac{x^2}{\rbr{x-\tilde{c}_0 \phi^{-1}}^2}\, dF_{\psi}(x)\\
    &= \psi\phi^{-1}\sbr{1+2\tilde{c}_0 \phi^{-1}s_{\psi}\rbr{\tilde{c}_0 \phi^{-1}} + \rbr{\tilde{c}_0 \phi^{-1}}^2s_{\psi}'\rbr{\tilde{c}_0 \phi^{-1}}}\\
    &= \frac{\phi - 2\phi^2\psi^{-1}+\phi\psi^{-1}}{1 - \phi^2\psi^{-1}}.
\end{align*}
Plugging the above equality into \eqref{equation:correlated_under_variance}, we  get
\begin{align*}
        V_{(S,X)}(\hat\beta^{S};\beta) = V_{(\beta, S,X)}(\hat\beta^{S};\beta) \to \sigma^2\rbr{\frac{\phi}{1-\phi}+\frac{\phi\psi^{-1}}{1-\phi\psi^{-1}} }.
\end{align*}

\subsubsection{Proof of Corollary \ref{coro:optimalsketchingmatrix}}

According to \eqref{equation:correlated_under}, we have
\begin{align*}
    1 = \psi \phi^{-1} + \psi \tilde{c}_0 \phi^{-2} \int \frac{1}{x - \tilde{c}_0 \phi^{-1}}\, dB(x) = \psi \phi^{-1}\rbr{1 + ts_B(t)},
\end{align*}
where $t = \tilde{c}_0 \phi^{-1}$ and $s_B(z) = \int \frac{1}{x-z}\, dB(x)$ is the Stieltjes transform of the measure $B$. Thus, we have $ts_B(t) = \psi^{-1}\phi -1$. In order to minimize \eqref{equation:correlated_under_variance}, it suffices to minimize the numerator $\psi\phi^{-1}\int \frac{x^2}{(t-x)^2} \, dB(x)$, which after simplification is $\psi\phi^{-1}\sbr{1+2ts_B(t)+t^2s_B'(t)}$. Therefore it suffices to minimize $t^2s_B'(t)$. By the Cauchy-Schwartz inequality, we have 
\begin{align*}
   t^2s_B'(t) = \int \frac{t^2}{(x-t)^2}\, dB(x) \geq  \rbr{\int \frac{t}{x-t}\, dB(x)}^2 = \rbr{\psi^{-1}\phi-1}^2,
\end{align*}
and  the minimum is achieved at $B = \delta_{\{a\}} (a>0)$.

\section{Proof of Theorem \ref{thm:correlated_deterministic}}\label{appendix:proof_nonrandom}
The proof to the variance part is the same as those for  Theorem \ref{thm:correlated_over} and Theorem \ref{correlated_under_biasvariance}. As for the bias part, when $p/m \to \phi\psi^{-1}<1$, same as \eqref{limit:correlated_under_bias}, it is easy to show almost surely for all large $n$, $B_{(\beta, S,X)}(\hat\beta^{S};\beta) = 0$. Hence, we only need to prove the bias part \eqref{limit:deterministic_bias2} for $p/m \to \phi\psi^{-1}>1$. Without loss of generality, we assume $\nbr{\beta} = 1$ throughout the proof. Let the SVD of $S$ be $S = U D V$ where $U \in \RR^{m \times m}$, $V \in \RR^{m \times n}$ are both orthogonal matrices, and  $D \in \RR^{m \times m}$ is a diagonal matrix. According to \eqref{equation:appendix_biaslemma_1} and \eqref{equation:appendix_biaslemma_2}, we have
\begin{align*}
    B_{(\beta, S,X)}(\hat\beta^{S};\beta) = \nbr{\Sigma^{1/2}\sbr{\rbr{X^\top  V^\top  VX}^{+}X^\top  V^\top  VX - I_p}\beta}^2.
\end{align*}
Let 
\begin{align*}
    h_n(z) = \nbr{\Sigma^{1/2}\sbr{\rbr{\frac{1}{p}X^\top  V^\top  VX - zI_p}^{-1} \frac{1}{p}X^\top  V^\top  VX - I_p}\beta}^2. 
\end{align*}
Then, for any $z<0$,
\begin{align}\label{ineq:appendix_deterministic_ridgeapprox}
    &\abr{B_{(\beta, S,X)}(\hat\beta^{S};\beta)^{1/2}- h_n(z)^{1/2}}\notag\\
    \leq~& \nbr{\Sigma^{1/2}\sbr{\rbr{\frac{1}{p}X^\top  V^\top  VX }^{+} \frac{1}{p}X^\top  V^\top  VX-\rbr{\frac{1}{p}X^\top  V^\top  VX-zI_p }^{-1} \frac{1}{p}X^\top  V^\top  VX}\beta}  \notag\\
    \leq~& \nbr{\Sigma^{1/2}}_2\nbr{\beta} \frac{\abr{z}}{\lambda_{\min}^{+}\rbr{\frac{1}{p}X^\top  V^\top  VX}-z} \notag\\
    \leq~& \frac{\sqrt{C_1}\abr{z}}{\lambda_{\min}^{+}\rbr{\frac{1}{p}X^\top  V^\top  VX}-z} \leq \frac{\sqrt{C_1}\abr{z}}{C_0\tau},
\end{align}
where the last inequality uses Lemma \ref{lemma:appendix_support} and \ref{lemma:appendix_sigmalambdamin}. By the fact that $\abr{B_{(\beta, S,X)}(\hat\beta^{S};\beta)}\leq C_1$ and  \eqref{ineq:appendix_deterministic_ridgeapprox}, we conclude 
\begin{align*}
    B_{(\beta, S,X)}(\hat\beta^{S};\beta) = \lim_{z\to 0^-}h_n(z) = \lim_{z\to 0^-}z^2\beta^\top  \rbr{\frac{1}{p}X^\top  V^\top  VX - zI_p}^{-1}\Sigma \rbr{\frac{1}{p}X^\top  V^\top  VX - zI_p}^{-1}\beta.
\end{align*}
Next, we follow the same idea as in the proof to the bias part in Theorem \ref{thm:isotropic_limiting_risk} to verify the interchange of the limits $n\to \infty$ and $z\to 0^-$. Since for any $z < 0$, $\abr{h_n(z)} \leq C_1$ and 
\begin{align*}
    &\abr{h_n'(z)}\\
    \leq& 2\nbr{\Sigma}_2\nbr{\beta}^2\nbr{z\rbr{\frac{1}{p}X^\top  V^\top  VX - zI_p}^{-1}}_2\nbr{\rbr{\frac{1}{p}X^\top  V^\top  VX - zI_p}^{-1} + z\rbr{\frac{1}{p}X^\top  V^\top  VX - zI_p}^{-2}}_2\\
    \leq& 2\frac{\lambda^+_{\min}\rbr{\frac{1}{p}X^\top  V^\top  V X} \nbr{\Sigma}_2}{\sbr{\lambda^+_{\min}\rbr{\frac{1}{p}X^\top  V^\top  V X}-z}^2} \leq \frac{2C_1}{C_0\tau},
\end{align*}
where the second and the last inequalities follow \eqref{ineq_fprime} and hold almost surely for all large $n$, we can exchange limits by Arzela-Ascoli theorem and Moore-Osgood theorem, that is, 
\begin{align}\label{limit:appendix_exchange_deterministic}
    \lim_{n\to \infty}B_{(\beta, S,X)}(\hat\beta^{S};\beta) = \lim_{n \to \infty}\lim_{z\to 0^-}h_n(z) = \lim_{z\to 0^-} \lim_{n\to \infty}h_n(z).
\end{align}
Next, we aim to find $\lim_{z\to 0^-} \lim_{n\to \infty}h_n(z)$. Let $\mathbb{D} = \{(z,w)\in \RR^2: z < 0, w > -\frac{1}{2C_1}\}$, 
\begin{align*}
    \mathcal{H}_n(z, w) = z\beta^\top  \rbr{\frac{1}{p}X^\top  V^\top  V X - zI_p - zw\Sigma}^{-1}\beta,
\end{align*}
which is defined on $\mathbb{D}$, and 
\begin{align*}
    \Sigma_w = \Sigma\rbr{I_p + w\Sigma}^{-1},\quad \beta_w = \rbr{I_p + w\Sigma}^{-1/2}\beta.
\end{align*}
Then, $\mathcal{H}_n(z,w)$ is analytic on $\mathbb{D}$ such that  
\begin{align*}
    \mathcal{H}_n(z,w) &= z\beta_w^\top  \rbr{\frac{1}{p}\Sigma_w^{1/2}Z^\top  V^\top  VZ \Sigma_w^{1/2} - zI_p}^{-1}\beta_w,\quad 
    h_n(z) = \frac{\partial \mathcal{H}_n}{\partial w}(z,0).
\end{align*}
Further write
\begin{align*}
    m_{1n}(z,w) = \frac{1}{p}\trace \sbr{\rbr{\frac{1}{p}\Sigma_w^{1/2}Z^\top  V^\top  V Z \Sigma_w^{1/2} - z I_p}^{-1}\Sigma_w}, ~ m_{2n}(z,w) = \frac{1}{p}\trace\sbr{\rbr{\frac{1}{p} V Z \Sigma_w Z^\top  V^\top  - z I_m}^{-1} }.
\end{align*}
According to \citep{paul2009no} or  \citep[Theorem 2.7]{couillet2021random}, $-\frac{1}{z}\rbr{I_p + m_{2n}(z,w)\Sigma_w}^{-1}$ is the  deterministic equivalent of $\rbr{\frac{1}{p}\Sigma_w^{1/2}Z^\top  V^\top  VZ \Sigma_w^{1/2} - zI_p}^{-1}$. Thus, it holds  that, for any given $(z,w)\in \mathbb{D}$, as $n \to \infty$,
\begin{align*}
    \mathcal{H}_n(z,w) - \tilde{\mathcal{H}}_n(z,w)  \to 0~~~\text{almost surely}, 
\end{align*}
where $\tilde{\mathcal{H}}_n(z,w)$ is defined as
\begin{align*}
    \tilde{\mathcal{H}}_n(z,w) = -\beta_w^\top  \rbr{I_p + m_{2n}(z,w)\Sigma_w}^{-1}\beta_w.
\end{align*}
Furthermore, it is easy to show that $\mathcal{H}_n(z,w)$ and $\tilde{\mathcal{H}}_n(z,w)$ are both uniformly bounded on $\mathbb{D}$. Thus, using the Vitali's theorem colleted as Lemma \ref{lemma:appendix_vitali}, we have for $z<0$ 
\begin{align}\label{limit:appendix_deterministic_hn}
    \lim_{n \to \infty}h_n(z) &= \lim_{n \to \infty} \frac{\partial \tilde{\mathcal{H}}_n}{\partial w}(z,0) \notag\\
    &= \lim_{n\to \infty} \rbr{1+\frac{\partial m_{2n}}{\partial w}(z,0)}\beta^\top  \sbr{1+ m_{2n}(z,0)\Sigma}^{-2}\Sigma \beta\notag\\
    &= \rbr{1+\frac{\partial m_{2}}{\partial w}(z,0)}\int \frac{x}{\sbr{1+m_2(z,0)x}^2} \, dG(x)
\end{align}
almost surely, 
where, according to \citep{zhang2007spectral}, 
$m_{1n}(z,w) \to m_1(z,w)$ and $m_{1n}(z,w) \to m_1(z,w)$ almost surely as $n\rightarrow \infty$. Moreover, for any given $(z,w) \in \mathbb{D}$, $(m_1(z,w), m_2(z,w)) \in \RR^{+}\times \RR^{+}$ is the unique solution to the self-consistent equations
\begin{align}\label{equations:appendix_deterministic_selfconsistent}
    m_1(z,w) &= \int \frac{x}{-z\sbr{1+wx+xm_2(z,w)}} \, dH(x), \notag\\
    m_2(z,w) &= \psi \phi^{-1} \frac{1}{-z\sbr{1+m_1(z,w)}}.
\end{align}
Substituting $m_2$ into $m_1$ in \eqref{equations:appendix_deterministic_selfconsistent} and using  $m_1(z,0)=m_1(z)$ as defined in \eqref{equations:appendix_selfconsistent}, we have after some calculations 
\begin{align*}
    \lim_{z\to 0 ^-}z\frac{\partial m_1}{\partial w}(z,0) = \frac{\int \frac{c_0^2 x^2}{\rbr{c_0-x\psi\phi^{-1}}^2} \, dH(x) }{1-\int \frac{ x^2\psi\phi^{-1}}{\rbr{c_0-x\psi\phi^{-1}}^2} \, dH(x)} = \phi\psi^{-1}c_1c_0^2,
\end{align*}
where  $c_0$ is defined in \eqref{equation:a} and $\lim_{z\to 0^-}zm_1(z) = c_0$. Using \eqref{limit:appendix_exchange_deterministic} and continue \eqref{limit:appendix_deterministic_hn}, we have almost surely
\begin{align*}
    \lim_{n\to \infty}B_{(\beta, S,X)}(\hat\beta^{S};\beta) &= \lim_{z\to 0^-} \lim_{n\to \infty}h_n(z)\\ 
    &= \lim_{z\to 0^-} \rbr{1+\frac{\partial m_{2}}{\partial w}(z,0)}\int \frac{x}{\sbr{1+m_2(z,0)x}^2} \, dG(x)\\
    &= \lim_{z\to 0^-}\rbr{1+\psi\phi^{-1}\frac{z\frac{\partial m_1}{\partial w}(z,0)}{\sbr{z + zm_1(z)}^2}}\int \frac{x}{\sbr{1-\psi\phi^{-1}x\frac{1}{z+zm_1(z)}}^2} \, dG(x)\\
    &= \rbr{1+c_1}\int \frac{c_0^2x}{\rbr{c_0-x\psi\phi^{-1}}^2} \, dG(x).
\end{align*}

\section{Proofs for central limit theorems}\label{appendix:sec:clt}

\subsection{Proof of Theorem~\ref{CLT_under}}
For the underparameterized case with $\phi\psi^{-1}<1$, by Lemma~\ref{lemma_bias_variance}, it holds that
\$
B_{(S,X)}(\hat \beta^S,\beta)  = 0,\quad
V_{(S,X)}(\hat \beta^S,\beta)  = \sigma^2 {\rm tr}\left\{(X^\top  S^\top  SX)^+ \right\}.
\$
Assume that $p<m<n$. Then $X^\top  S^\top  SX$ is of rank $p$ and then invertible. So
\$
R_{(S,X)}\rbr{\hat{\beta}^S;\beta} & = V_{(S,X)}(\hat \beta^S,\beta) = \sigma^2 {\rm tr}\left\{(X^\top  S^\top  SX)^{-1} \right\}
 = \frac{\sigma^2}{p}{\rm tr}\left\{ (X^\top  S^\top  SX/p)^{-1}\right\}\\
& = \sigma^2 \int \frac{1}{t}d F^{X^\top  S^\top  SX/p}(t)
=:  \sigma^2 s_{1n}(0),
\$
where $s_{1n}(\cdot)$ denotes the Stieltjes transformation of the ESD $F^{X^\top  S^\top  SX/n}$ of ${X^\top  S^\top  SX/n}\in\RR^{p\times p}$. 

Let $\underline B_n$ denote the ESD of $S^\top  S \in \RR^{n\times n}$ and $\underline B$ its LSD. 
Define 
\$
Q_n:= \frac{1}{p}(S^\T  S)^{1/2}XX^\T  (S^\T  S)^{1/2}\in\RR^{n\times n}.
\$
Under Assumptions~\ref{CLT_assumption}, the matrix $Q_n$ has the LSD $F^{\phi^{-1},\underline B}$, which is the Marcehnko-Pastur law. 
Further, we define
\$
\cG_n(t) : &  = n \left\{F^{Q_n} (t) - F^{\phi_n^{-1},\underline B_n}(t) \right\},
\$
where we use $F^{\phi_n^{-1},\underline B_n}$ instead of  $F^{\phi^{-1},\underline B}$ to avoid discussing the convergence of $(\phi_n^{-1},\underline B_n)$ to $(\phi^{-1},\underline B)$. For orthogonal sketching, $\underline B_n = (1-\psi_n) \delta_0 + \psi_n\delta_1$ and $\underline B = (1-\psi) \delta_0 + \psi\delta_1 $. 
Notice that
\$
\cG_n(t) = p \left\{F^{X^\top S^\top  SX/p} (t) - \underline F^{\phi_n^{-1},\underline B_n}(t) \right\}
\$
with $\underline F^{\phi_n^{-1},\underline B_n}:= (1-\phi_n^{-1})\delta_0 +\phi_n^{-1}F^{\phi_n^{-1},\underline B_n}$. By Theorem~\ref{thm:isotropic_limiting_risk}, we have 
\$
R_{(S,X)}(\hat \beta^S,\beta) \overset{a.s.}{\longrightarrow} \frac{\sigma^2\phi\psi^{-1}}{1-\phi\psi^{-1}}.
\$
Further, we can rewrite 
\#\label{process_1}
p\left(R_{(S,X)}(\hat \beta^S,\beta) -  \frac{\sigma^2\phi_n\psi_n^{-1}}{1-\phi_n\psi_n^{-1}}\right)
& = \sigma^2 \int \frac{1}{t} d \cG_n(t),
\#
where we replaced $(\phi,\psi)$ by $(\phi_n,\psi_n)$ when centering. \\

We prove the CLT for  \eqref{process_1} in the following two steps. 

{\it Step 1.} Given the sketching matrix $S$, by \citep[Theorem~2.1]{zheng2015substitution}, the RHS of \eqref{process_1} converges to a Gaussian distribution with mean $\mu_1$ and variance $\sigma_1^2$ specified as 
\$
\mu_1= & -\frac{\sigma^2}{2\pi i}\oint_\cC \frac{1}{z}
\frac{\phi^{-1}\int \underline s_{\phi}(z)^3 t^2 (1+t \underline s_{\phi}(z))^{-3}d \underline B(t)}{\left\{1-\phi^{-1}\int \underline s_{\phi}^2 t^2 (1+t\underline s_{\phi}(z))^{-2}d \underline B(t)\right\}^2} d z \\
& - \frac{\sigma^2 (\nu_4-3)}{2\pi i}\oint_\cC \frac{1}{z}
\frac{\phi^{-1}\int \underline s_{\phi}(z)^3 t^2 (1+t \underline s_{\phi}(z))^{-3}d \underline B(t)}{1-\phi^{-1}\int \underline s_{\phi}^2 t^2 (1+t\underline s_{\phi}(z))^{-2}d \underline B(t)} d z 
\$
and 
\$
\sigma_1^2 = &  -\frac{2\sigma^4 }{4 \pi^2}\oint_{\cC_1}\oint_{\cC_2} \frac{1}{z_1 z_2}\frac{1}{(\underline s_\phi(z_1)- \underline s_{\phi}(z_2))^2} d \underline s_\phi(z_1) d \underline s_\phi(z_2)\\
& - \frac{\phi^{-1}(\nu_4-3)}{4\pi^2} \oint_{\cC_1}\oint_{\cC_2} \frac{1}{z_1 z_2} \left\{\int \frac{t}{(\underline s_\phi(z_1)+1)^2} \frac{t}{(\underline s_\phi(z_2)+1)^2} d \underline B(t)\right\} d \underline s_\phi(z_1) d \underline s_\phi(z_2),
\$
where $\underline s_{\phi}(\cdot)$ denotes the Stieltjes transformation of $\underline F^{\phi^{-1},\underline B}$, and  $\cC$, $\cC_1$ and $\cC_2$ are contours containing the support of the LSD of $\underline s_{\phi}(z)$.

Following the calculation of $\mu_c$ and $\sigma_c^2$ in the proof of \citep[Theorem~4.1]{li2021asymptotic}, we get
\$
\mu_1 & = \frac{\sigma^2 \phi^2\psi^{-2}}{(\phi\psi^{-1}-1)^2}+\frac{\sigma^2\phi^2\psi^{-2}(\nu_4-3)}{1-\phi\psi^{-1}},\quad
\sigma_1^2  =\frac{2\sigma^4\phi^3\psi^{-3}}{(\phi\psi^{-1}-1)^4}+\frac{\phi^3\psi^{-3}\sigma^4(\nu_4-3)}{(1-\phi\psi^{-1})^2}.
\$

{\it Step 2.} Note that the mean $\mu_1$ and variance $\sigma^2_1$ are nonrandom.  It means that the limiting distribution of the RHS of \eqref{process_1} is independent of conditioning $SS^\top $. So it asymptotically follows the Gaussian distribution $\cN(\mu_1,\sigma_1^2)$.

\subsection{Proof of Theorem~\ref{CLT_over}}
For the overparameterized regime with $\phi\psi^{-1}>1$, when $\Sigma=I_p$, we have
\$
B_{(S,X)}(\hat \beta^S,\beta) & = \frac{\alpha^2}{p}{\rm tr}\left\{I_p -(X^\top  S^\top  S X)^+ X^\top  S^\top  S X \right\}
 = \alpha^2 -\frac{\alpha^2}{p}{\rm tr}\left\{(X^\top  S^\top  S X)^{+} X^\top  S^\top  S X \right\}\\
& = \alpha^2 -\frac{\alpha^2}{p}{\rm tr}\left\{(S X X^\top  S^\top  )^{+} S X X^\top  S^\top \right\}
= \alpha^2 \left(1-\phi^{-1}\psi_n\right)
\$
and 
\$
V_{(S,X)}(\hat \beta^S,\beta) & = \sigma^2 {\rm tr}\left\{(X^\top  S^\top  SX)^{+} \right\}
= \sigma^2 \frac{1}{n}{\rm tr}\left\{(S X X^\top  S^\top /n)^{-1} \right\}\\
& = \sigma^2 \int \frac{1}{t}d F^{SXX^\top  S^\top  /n}(t) 
=: \sigma^2 \underline s_{1n}(0),
\$
where $\underline s_{1n}(z)$ denotes the Stieltjes transformation of the ESD $F^{SXX^\top  S^\top /n}$ of $SXX^\top  S^\top  /n$ and it satisfies
\$
\frac{p}{m}s_{1n}(z) = -\frac{1}{z} (\frac{p}{m}-1)+\underline s_{1n}(z). 
\$
Following the proof of \citep[Theorem~4.3]{li2021asymptotic}, we  get 
\$
\mu_2 & = \frac{\sigma^2 \phi\psi^{-1}}{(\phi\psi^{-1}-1)^2}+\frac{\sigma^2 (\nu_4-3)}{\phi\psi^{-1}-1},\quad
\sigma_2^2  =\frac{2\sigma^4\phi^3\psi^{-3}}{(\phi\psi^{-1}-1)^4}+\frac{\sigma^4\phi\psi^{-1}(\nu_4-3)}{(\phi\psi^{-1}-1)^2}.
\$


{
}

\subsection{Proof of Theorem \ref{thm:clt_over_conditioningbeta}}
Leveraging \citep[Theorem 7.2]{bai2008central} or following a similar proof to that of \citep[Theorem 4.5]{li2021asymptotic}, we   obtain
\begin{align}
    &\sqrt{p}\cbr{B_{(\beta, S,X)}(\hat\beta^{S};\beta) - \alpha^2(1-\phi_n^{-1}\psi_n)} \notag \\ =&  \sqrt{p}\cbr{\beta^\top  \sbr{I_p - \rbr{  X^\top  V^\top  VX}^{+}X^\top  V^\top  VX } \beta - \alpha^2(1-\phi_n^{-1}\psi_n)} \overset{D}{\longrightarrow} \cN(0, d^2 = d_1^2 + d_2^2), \notag
\end{align}
where 
\begin{align*}
  d_1^2 = wp^2\rbr{ \EE(\beta_1^4) - \gamma^2},\ d_2^2 = 2p^2(\theta-w)\gamma^2  
\end{align*}
and 
\begin{align*}
\gamma = \EE(\beta_1^2) = \frac{\alpha^2}{p},\ \theta = \lim_{p\to \infty}\frac{1}{p} {\rm tr}\left\{I_p -(X^\top  S^\top  S X)^+ X^\top  S^\top  S X \right\}  =    1- \phi^{-1}\psi.
\end{align*}
Here $w$ is the limit of the average of squared diagonal elements of $\sbr{I_p - \rbr{  X^\top  V^\top  VX}^{+}X^\top  V^\top  VX }$, which will be canceled out in $d^2$ under the assumption that $\beta$ is multivariate normal. After some simple calculation, we have $d^2 = 2(1-\phi^{-1}\psi)\alpha^4$. Thus, 
\begin{align}\label{equal_clt_fixedbeta_bias}
    \sqrt{p}\cbr{B_{(\beta, S,X)}(\hat\beta^{S};\beta) - \alpha^2(1-\phi_n^{-1}\psi_n)} \overset{D}{\longrightarrow} \cN (0, 2(1-\phi^{-1}\psi)\alpha^4).
\end{align}
Moreover, we have proved in the Theorem \ref{CLT_over} that
\begin{align}\label{equal_clt_fixedbeta_var}
 p\left(V_{(\beta, S,X)}(\hat \beta^S;\beta)-\frac{\sigma^2}{\phi_n\psi_n^{-1}-1}\right)
\overset{D}{\longrightarrow} \cN (\mu_2, \sigma^2_2),
\end{align}
where 
\$
\mu_2 & = \frac{\sigma^2 \phi\psi^{-1}}{(\phi\psi^{-1}-1)^2}+\frac{\sigma^2 (\nu_4-3)}{\phi\psi^{-1}-1},\quad
\sigma_2^2  =\frac{2\sigma^4\phi^3\psi^{-3}}{(\phi\psi^{-1}-1)^4}+\frac{\sigma^4\phi\psi^{-1}(\nu_4-3)}{(\phi\psi^{-1}-1)^2}.
\$
Combining \eqref{equal_clt_fixedbeta_bias} and \eqref{equal_clt_fixedbeta_var} completes the proof. 

\end{document}